\documentclass[11pt]{amsart}
\usepackage{amsmath}
\usepackage{amsfonts,amssymb,amscd,amsthm,comment,cite,float}
\usepackage[thinlines]{easytable}
\usepackage{array}
\usepackage[shortlabels]{enumitem}
\usepackage{mathtools}
\usepackage{hyperref}

\usepackage{color, soul}
\definecolor{gr}{rgb}{0.7, 1, 0.7}
\definecolor{rr}{rgb}{1, 0.7, 0.7}

\DeclareMathSymbol{\shortminus}{\mathbin}{AMSa}{"39}
\DeclareMathSymbol{\shortplus}{\mathbin}{AMSa}{"2B}
\DeclareMathOperator{\sgn}{sgn}

\usepackage{stackengine}

\usepackage{latexsym}
\usepackage{graphicx}

\theoremstyle{plain} %definition remark
\newtheorem{theorem}{Theorem}[section]
\newtheorem*{mainthm}{Main theorem}
\newtheorem*{FKRW}{Frechet-Kolmogorov-Riesz-Weil Compactness Theorem}
\newtheorem{lemma}[theorem]{Lemma}
\newtheorem{corollary}[theorem]{Corollary}

\newtheorem{proposition}[theorem]{Proposition}

\theoremstyle{definition} %definition remark

\theoremstyle{remark} %definition remark

\newcolumntype{C}[1]{>{\centering\arraybackslash$}p{#1}<{$}}

\renewcommand{\mathfrak}{\mathbf}

\newcommand{\ignore}[1]{}

\usepackage{bm}

\newcommand{\cE}{{\mathcal E}}
\newcommand{\cF}{{\mathcal F}}

\newcommand{\cI}{{\mathcal I}}

\newcommand{\cL}{{\mathcal L}}
\newcommand{\cM}{{\mathcal M}}
\newcommand{\cN}{{\mathcal N}}

\newcommand{\cR}{{\mathcal R}}

\newcommand{\RR}{{\mathbb R}}
\newcommand{\RRPP}{{\mathbb RP}}

\newcommand{\TT}{{\mathbb T}}

\newcommand{\NN}{{\mathbb N}}

\def\epsilon{\varepsilon}

\def\ii{\mathrm{i}}
\def\ee{\mathrm{e}}

\def\PV{\mathrm{PV}\int}

\def\dif{ {\mbox{\rm d}} }
\def\pd{ \partial }
\title[Renormalization and blow-up solutions for 1D Navier-Stokes]{Renormalization and existence of finite-time blow up solutions for a one-dimensional analogue of the Navier-Stokes equations}
\author{Denis Gaidashev}
\address{Uppsala University, Uppsala, Sweden}
\email{gaidash@math.uu.se}

\author{Alejandro Luque}
\address{Comsol, Inc., Stockholm, Sweden}
\email{alejandro.luque.math@gmail.com}
\thanks{{\bf This research was supported by Kurt and Alice Wallenberg Foundation grant $\mathbf{2015.0365}$}}

\subjclass[2010]{}
\begin{document}
\begin{abstract}
  The one-dimensional quasi-geostrophic equation is the one-dimensional Fourier-space analogue of the famous Navier-Stokes equations. In \cite{LS0} Li and Sinai  have proposed a renormalization approach to the problem of existence of finite-time blow up solutions of this equation. %In this setting, existence of finite time blow ups is a consequence of existence of a fixed point for a certain renormalization operator on an appropriate functional space.  \cite{LS0} gives a proof of existence of  complex-valued finite time blow up solutions of the quasi-geostrophic equation.
  In this paper we revisit the renormalization problem for the quasi-geostrophic blow ups, prove existence of a family of renormalization fixed points, and deduce existence of real $C^\infty([0,T),C^\infty(\RR) \cap L^2(\RR))$ solutions to the quasi-geostrophic equation whose energy and enstrophy become unbounded in finite time, different from those found in \cite{LS0}.
\end{abstract}
\maketitle

%\tableofcontents

\section{Introduction} Bakhtin, Dinaburg and Sinai \cite{BDS} proposed a novel approach to the question of existence of the blow-up solutions in the Navier-Stokes equations.  In their approach, a solution of an initial value problem with a certain self-similar initial conditions is ``shadowed'' by a solution of a fixed point problem for an integral nonlinear ``renormalization'' operator in an appropriate functional space. The self-similarity of the solution is of the type
\begin{equation}
\label{Lerayss} u(x,t)=(T-t)^{-{1 \over 2}} U \left((T-t)^{-{1 \over 2}} x \right).
\end{equation}
The question of existence of blow-up solutions of this form has been first addressed by J. Leray in \cite{Leray}, conditions for the triviality of such solutions have been studied in \cite{Tsai}, \cite{NRS}. In this regard, Bakhtin-Dinaburg-Sinai renormalization can be viewed as a technique of proving existence of Leray self-similar solutions through a fixed point theorem.

Later this renormalization approach was applied to several hydrodynamics models by Li and Sinai \cite{LS0,LS1,LS2,LS3}. The method of these publications is a development of that of  \cite{BDS}, with the addition that the authors derive an exact, and not an approximate, renormalization fixed point equation. %: existence of a fixed point implies existence of finite time blow up solutions for the hydrodynamic models considered. % Thus obtained solutions are, however, complex-valued.

In this paper we study existence of Leray solutions in the 1D quasi-geostrophic equation 
\begin{equation}\label{eq:modelB}
\pd_t u - (H u \cdot u)_x = u_{xx}\,,
\end{equation}
where $Hu$ is the usual Hilbert transform:
\begin{equation}\label{eq:Hilbert}
Hu(y):=\frac{1}{\pi} \PV_{-\infty}^\infty \frac{u(x)}{y-x} \dif x\,.
\end{equation}
The importance of this equation is that in the Fourier space it serves as a 1D analogue of the Navier-Stokes equations. %It has been shown in \cite{LS0} that $(\ref{eq:modelB})$ admits complex-valued solutions  that blow up in finite time.

We will derive an equation for the Fourier transform of certain self-similar solutions of  $(\ref{eq:modelB})$ and demonstrate that this equation reduces to a fixed point problem for a {\it renormalization} operator. We then prove existence of a family of  even and exponentially decaying renormalization fixed points, which correspond to  real  $C^\infty$-solutions of  $(\ref{eq:modelB})$ which blow up in finite time $T$.

Our approach to  the renormalization is different from that of Li and Sinai: rather than looking for solutions in the form of a power series and deriving a fixed point equation for the series' coefficients, we use considerations of compactness borrowed from the renormalization theory in dynamics to prove existence of a class of real solutions different from \cite{LS0}. %, whose speed of blow up in time also differs from the solutions obtained in \cite{LS0}.

\subsection{Outline of the proof} As usual, a solution to a Picard-Lindel\"off integral equation associated to a PDE, will be called a {\it mild} solution.

In Section  $\ref{secrenorm}$ we will demonstrate that the problem of existence of Leray mild self-similar blow ups is equivalent to a problem of existence of a common $L^2$-fixed point of a family of integral operators $\cR_\beta$ acting on functions of one variable, with  $\beta \in (0,1)$ playing the role of ``time'': $\beta^2=(T-t)/T$.

In Section $\ref{secapriori}$ this fixed point problem will be reduced to a construction of relatively compact sets $\cN_p$  in $L^p$, $p \ge 1$, invariant by all $\cR_\beta$ with $\beta \in (\beta_0,1)$  for some $\beta_0 \in (0,1)$. The sets $\cN_p$ will be constructed as the non-empty intersections of two convex sets $\cE^p_{a,k,K,A,\delta_0}$ and $\cM_{\mu,\sigma}$. The first, $\cE^p_{a,k,K,A,\delta_0}$, is the set of all $\psi \in L^p(\RR)$ satisfying the following convex conditions:

\vspace{2.1mm}

\noindent  ${\bf 1)}$  $|\psi(\eta)| \le k e^{-a|\eta|}$;

\vspace{2.1mm}

\noindent  ${\bf 2)}$  $|\psi(\eta)-\psi(\eta-\delta)| \le \omega_\psi(\eta) |\delta|^\alpha$ with $\omega_\psi>0$, $\omega_{\psi} \in L^p(\RR)$, $\| \omega_\psi \|_p \le K$ and $\omega_\psi(\eta) \le A |\eta| e^{-a |\eta|}$ whenever $|\eta| >1$.

\vspace{2.1mm}

The second set $\cM_{\mu,\sigma}$ is the set of all  $\psi$ in $L^1_u(\RR)$, where $u$ is the weight $u=|\eta|^\sigma e^{-\eta^2}$, such that the following integral is larger than  some $\mu>0$:
\begin{equation}
  \nonumber \cI[\psi]=\int_\RR \psi(\eta)|\eta|^\sigma e^{-\eta^2} \ \dif \eta.
\end{equation}

The invariance of the compact closure of the convex set  $\cN_p=\cE^p_{a,k,K,A,\delta_0} \cap \cM_{\mu,\sigma}$ under $\cR_\beta$ results, via Tikhonov fixed point theorem, in existence of fixed points $\psi_{\beta,p} \in \overline{\cN_p}$ of $\cR_\beta$ for every $\beta \in (\beta_0,1)$. At the same time, as already mentioned, existence of mild  blow ups requires existence of one and the same fixed point $\psi_p$ for all  $\cR_\beta$ with $\beta \in (0,1)$.  We show in Section $\ref{seclimits}$ that any limit along a subsequence of the fixed points $\psi_{\beta_i,p}$ with $\beta_i \in (\beta_0,1)$ is a fixed point of $\cR_\beta$ for all $\beta \in (0,1)$. Together with the fact that $\overline \cN_p \subset L^2(\RR)$, for all $p \ge 1$, this completes the proof of existence of Leray blow ups.

Additionally, in Section $\ref{secfp}$, we compute the common point $\psi_p$ as the inverse  Fourier transform of the inverse Laplace transform of an explicit expression.
\subsection{Statement of results}
The main result of this paper reads as follows:
\begin{mainthm}
There exist $\nu \in (0,1)$ such that for any $T>0$ there is a solution of equation $(\ref{eq:modelB})$  in the class $C^\infty\left([0,T), C^\infty(\RR) \bigcap \left( \bigcap_{m \ge 2} L^m(\RR) \right)\right)$, given by 
    $$u_\nu(x,t)=  (T-t)^{-{1 \over 2}} {\cF^{-1} \circ \cL^{-1}[\hat \vartheta_\nu]\left(x (T-t)^{-{1 \over 2}}   \right) },$$
 where  $\cF^{-1}$ denotes the inverse Fourier transform, and $\cL^{-1}$ - the inverse Laplace transform, and $\hat \vartheta_{\nu}$ is as a solution of the equation
    \begin{equation}
      \label{Riccati} \hat \vartheta'(s)=-\hat \vartheta(s)^2-{s \over 2} \hat \vartheta(s)+{\nu \over 2},
\end{equation}
    such that the energy $E[u_\nu](t):=\| u_\nu(\cdot,t) \|_2^2$ and the enstropy $\Omega[u_\nu](t):=\| w(\cdot, t) \|_2^2$ where $w(x,t):=x u(x,t)$, become unbounded as $t \rightarrow T$.   
\end{mainthm}

\subsection{The Navier-Stokes equations.} Recall the Navier-Stokes equations:
\begin{equation}\label{eq:NS}
\partial_t u + (u \cdot \nabla) u - \Delta u + \nabla p = 0\,, 
\qquad
\nabla \cdot u =0\,, 
\qquad
u(\cdot,0)=u_0,
\end{equation}
where  $u_0$ is the initial datum, and  $u$ and $p$ are the unknowns.
%where  $u_0: \RR^d \rightarrow \RR^d$ is the initial datum, and  $u : \RR^d \times \RR \rightarrow \RR^d$ and $p : \RR^d \times \RR \rightarrow \RR$ are the unkowns.

We apply the Fourier transform  $v(y)=\cF[u](y)$, %given by
%\[
%v(y,t)=\int_{\RR^d} u(x,t) \ee^{-\ii y \cdot x} \dif x\,,
%\qquad
%u(x,t)=\int_{\RR^d} v(y,t) \ee^{\ii y \cdot x} \dif y\,,
%\]
so that the Navier-Stokes PDE becomes the following system of integro-differential equations:
\begin{equation}\label{eq:sist:FT}
\frac{\pd v(y,t)}{\pd t} = -|y|^2 v(y,t) + \ii \int_{\RR^d} (y \cdot v(y-z,t)) P_y v(z,t) \dif z \,,
\qquad
\end{equation}
where $P_y$ is the Leray projection to the subspace orthogonal to $y$, i.e.
\begin{equation}\label{eq:leray}
P_y v = v - \frac{v \cdot y}{|y|^2} y\,,
\end{equation}
indeed,  incompressibility condition $\nabla \cdot u=0$ reads $v(y,t) \cdot y =0$ in Fourier space. % Notice, that if $u : \RR^d \times \RR \rightarrow \RR$ then $v:\RR^d\times \RR \rightarrow \CC$ satisfies the symmetry property $v(y,t)=\overline{v(-y,t)}$ for all $t \in \RR$.
Integrating \eqref{eq:sist:FT} with respect to $t$, we obtain the integral equation:
\begin{equation}\label{eq:integral:NS}
v(y,t) = \ee^{-|y|^2 t} v_0(y) + \ii \int_0^t \int_{\RR^d} (y \cdot v(y-z,\tau)) P_y v(z,\tau)
\ee^{-|y|^2 (t-\tau)} \dif z \dif \tau\,,
\end{equation}
where $v_0=v(y,0)$ stands for the initial datum at $t=0$.

\subsection{The quasi-geostrophic equation}
Our primary focus will be  equation $(\ref{eq:modelB})$, which according to the discussion of \cite{CCCF} is a 1D version of the 2D quasi-geostrophic initial value problem:
\begin{align} 
\label{QC1} &\theta_t + (u \cdot \nabla) \theta=0, \\
%\label{QC2}  && u=\nabla^{\perp} \psi, \quad \nabla^{\perp} := (-\partial_2, \partial_1), \\
%\label{QC3}  && \theta=-(-\Delta)^{1 \over 2} \psi, \\
\label{QC2} &u= - R^{\perp} \theta:=(-R_2 \theta, R_1 \theta),\\
\label{QC4}  & \theta(x,0)=\theta_0(x).
\end{align}
where $R_j$ is the Riesz transform:
$$R_j(\theta)(x,t)={1 \over 2 \pi} PV \int_\RR^{2} {(x_j-y_j) \theta(y,t) \over |x-y|^{3}  } \dif y.$$
This system of equations models the evolution of the temperature in a mixture of cold and hot air fronts. 

The system $(\ref{QC1})-(\ref{QC2})$ can be alternatively written as
$$\theta_t + {\rm div} [ (R^{\perp} \theta) \theta ]=0.$$
One obtains the 1D quasi-geostrophic equation by assuming that $\theta(x,t) \equiv \theta(x_1,t)$: integration of  the Riesz transform with respect to the variable $y_2$ results in a Hilbert transform of $\theta(x_1,t)$.
%The Hilbert transform is well characterized in terms of Fourier transform. Indeed, it can be equivalently defined as
%\[
%\cF[H u](y)=-\ii \sgn(y) \cF[u](y)\,.
%\]
%where
%\[
%\sigma(y) =
%\left\{
%\begin{array}{rr}
%1 & \mbox{if $y > 0$} \,, \\
%0 & \mbox{if $y=0$} \,, \\
%-1 & \mbox{if $y< 0$} \,.
%\end{array}
%\right.
%\]

The model $(\ref{eq:modelB})$ has been considered in the classical references \cite{CLM,DG,Schochet,MajdaB02}.  \cite{Mats} and \cite{BLM} consider the periodic case, and proof existence of finite-time blow-up $C^{\infty}(\TT \times [0,T))$-solutions for this equation.  Li and Sinai show in \cite{LS0} that $(\ref{eq:modelB})$ admits complex-valued solutions  that blow up in finite time.

The integral equation for this model, analogous to~\eqref{eq:integral:NS}, 
is given by
\begin{equation}\label{eq:integral:ModelB}
v(y,t) = \ee^{-|y|^2 t} v_0(y) + \hspace{-0.6mm} \int_0^t \hspace{-1.2mm} \int_\RR \hspace{-1.1mm}y v(y-z,\tau) \sgn(y-z) v(z,\tau)  \ee^{-|y|^2 (t-\tau)} \dif z\dif \tau\,.
\end{equation}
%The Leray projection is  meaningless for $d=1$: it is identically equal to zero.
The operator
\begin{equation}\label{eq:leray:d1}
  P_y v(z) = - \ii \sgn(y-z) v(z),
\end{equation}
appearing in $(\ref{eq:integral:ModelB})$ is also a projection. In Section $\ref{secrenorm}$ we will derive a renormalization operator for  a generalized version of $(\ref{eq:integral:ModelB})$ with an exponent $\gamma>1$:
\begin{equation}\label{eq:integral:NSgamma}
v(y,t) = \ee^{-|y|^\gamma t} v_0(y) + \ii \int_0^t \int_{\RR} y v(y-z,\tau) P_y v(z,\tau) \ee^{-|y|^\gamma (t-\tau)} \dif z \dif \tau\,. 
\end{equation}

\section{Renormalization problem for a self-similar solution}\label{secrenorm}
In this section we will consider an initial value problem for $(\ref{eq:integral:NSgamma})$  with a Leray initial condition, and deduce that existence of a Leray solution to $(\ref{eq:integral:NSgamma})$ is equivalent to existence  of a fixed point for a certain operator.

We will make an assumption that
\begin{equation}\label{vform}
  v(y,t)= \tau(t)^{-c} \psi(y \tau(t)),
\end{equation}
where $\psi$ is some bounded function, and $\tau(t)=(T-t)^{1/ \gamma}$.  The following is equation   $(\ref{eq:integral:NSgamma})$ for the initial value problem with the initial condition
$$v_0(y)= \tau_0^{-c} \psi(y \tau_0)=T^{-{c \over \gamma}} \psi(y T^{1 \over \gamma} ), \quad \tau_0 \equiv \tau(0)$$
{\it under the assumption that  $v(y,t)$ remains of the form $(\ref{vform})$ for $t > 0$}:
\begin{equation*}
{ \psi(y \tau(t)) \over \tau(t)^{c}} = e^{-|y|^\gamma t} { \psi(y \tau_0) \over \tau_0^{c}} + y \hspace{-0.6mm}  \int_{0}^t  \hspace{-0.6mm} \int_{0}^y   \hspace{-0.6mm}  e^{-|y|^\gamma (t-s) }  { \psi(z \tau(s))  \psi((y-z) \tau(s)) \over  \tau(s)^{2 c} } \  \dif z \ \dif s,
\end{equation*}
or, denoting the convolution $\int_0^y \psi(x) \psi(y-x) \ \dif x$ as $ \psi \bullet \psi(y)$, % this equation  can be written as 
\begin{equation}
  \nonumber   {\psi(y \tau(t)) \over \tau(t)^c}  =  e^{\shortminus |y|^\gamma (\tau_0^\gamma \shortminus \tau(t)^\gamma) } {\psi(y \tau_0) \over \tau_0^{c}} \hspace{0.2mm} \shortminus \hspace{0.2mm} \gamma y \hspace{-1mm} \int_{\tau_0}^{\tau(t)}  \hspace{-4.5mm}  e^{\shortminus |y|^\gamma (\tau(s)^\gamma \shortminus \tau(t)^\gamma) }  { (\psi  \bullet \psi)(y \tau(s) )  \over  \tau(s)^{2 c+2-\gamma} }   \ \dif \tau(s).
\end{equation}
Set $\eta=y \tau_0, \quad \xi= y \tau(t), \quad \zeta=y \tau(s)$, then 
\begin{eqnarray}
  \nonumber    {\psi(\xi) \over \xi^c}&=&  e^{|\xi|^\gamma-|\eta|^\gamma} {\psi(\eta) \over \eta^{c}}+ \gamma \int_{\xi}^\eta   e^{\xi^\gamma-\zeta^\gamma}  { (\psi \bullet \psi)(\zeta)  \over  \tau(s)^{2 c+2-\gamma} y^{c}  }   \ \dif \zeta.
\end{eqnarray}
We can now make the choice of $c=\gamma-2$, after which we get
\begin{equation}
  \label{problem1}   \psi(\xi)=  e^{\xi^\gamma-\eta^\gamma} {\xi^{\gamma-2} \over \eta^{\gamma-2}} \psi(\eta)+ \gamma \  \xi^{\gamma-2} e^{\xi^\gamma} \int_{\xi}^\eta   e^{-\zeta^\gamma}  \zeta^{2-\gamma} (\psi \bullet \psi)(\zeta)   \ \dif \zeta.
\end{equation}
A function $\psi$ solving this equation for all real $\eta$ and $\xi$ such that $|\xi| \le |\eta|$, provides a solution
$$v(y,t)=  (T-t)^{2-\gamma \over \gamma} \psi\left(y (T-t)^{1 \over \gamma} \right)$$
to the initial value problem $(\ref{eq:integral:NSgamma})$  with the initial data
\begin{equation}
\label{IC}v_0(y)= T^{2-\gamma \over \gamma} \psi\left(y T^{1 \over \gamma} \right).
\end{equation}
By introducing a new time variable,
\begin{equation}
\label{beta_def}\beta=T^{-{1 \over \gamma}} (T-t)^{1 \over \gamma},
\end{equation}
equation $(\ref{problem1})$ can be written as
\begin{equation}
  \label{problem2}   \psi(\eta \beta)= \beta^{\gamma-2} \hspace{-0.6mm}  \left( \hspace{-0.5mm} e^{\beta^{\gamma} \eta^\gamma-\eta^\gamma} \psi(\eta) \hspace{-0.5mm}  +  \hspace{-0.5mm} \gamma \eta^{\gamma-2} e^{\beta^\gamma \eta^\gamma} \hspace{-2mm} \int_{\beta \eta }^\eta  \hspace{-1mm} e^{-\zeta^\gamma}  \zeta^{2-\gamma} (\psi \bullet \psi)(\zeta)   \ \dif \zeta \right)\hspace{-0.7mm},
\end{equation}
while the problem of existence of solution of $(\ref{problem2})$ can be stated as a fixed point problem for a family of operators
\begin{equation}
  \label{operator} \cR_\beta[\psi](\eta)=  \beta^{\gamma\shortminus 2}  e^{\shortminus\left({1 \over \beta^{\gamma}}\shortminus 1\right) \eta^\gamma} \hspace{-1.5mm} \psi\left({\eta \over \beta} \right)+ \gamma \eta^{\gamma\shortminus 2} e^{\eta^\gamma}  \hspace{-1.5mm} \int_{\eta }^{\eta \over \beta}   \hspace{-1.5mm}  e^{\shortminus \zeta^\gamma}  \zeta^{2\shortminus \gamma} (\psi \bullet \psi)(\zeta)   \ \dif \zeta.
\end{equation}
In the next sections we will demonstrate that in the most interesting case $\gamma=2$, the operator $(\ref{operator})$ has a fixed point $\psi_{\beta,p}$ for every $p \ge 1$ and for each $\beta \in (\beta_0,1)$ for some $\beta_0 \in (0,1)$, in a certain subset $\cN_p \subset L^p(\RR) \cap L^2(\RR)$  of measurable exponentially decaying functions. In particular, the inverse Fourier transform, $\cF^{-1}[\psi](x)=\cF[\psi](-x)$ is a well defined operator from $\cN_p$ to $L^2(\RR)$, in fact, to $C^{\infty}(\RR)$. We will eventually show that the function
\begin{equation*}
  u(x,t)= (T-t)^{ -{1 \over 2}}  \cF^{-1}[\psi_p]\left({x   (T-t)^{-{1 \over 2}} }\right)
\end{equation*}
is $C^\infty(\RR \times [0, T))$ solution of $(\ref{eq:modelB})$. For such  solution both the energy and the enstropy become unbounded in finite time. By Plancherel's theorem:
\begin{align}
  %\label{energy}  E[u](t) &= \int_{\RR} |v(y,t)|^2 \dif y = \int_{\RR} \psi(y (T-t)^{1 \over 2} )^2 \dif y  =  {1 \over (T-t)^{1 \over 2}} \int_{\RR} \psi(\zeta)^2 \dif \zeta  \\
  \label{energy}  E[u](t) &= \int_{\RR} |v(y,t)|^2 \dif y = {1 \over (T-t)^{1 \over 2}} \int_{\RR} \psi(\zeta)^2 \dif \zeta,  \\
\label{enstropy}  \Omega[u](t)&=  \int_{\RR^d} |y|^2 |v(y)|^2 \dif y =  {1 \over (T-t)^{3 \over 2}} \int_{\RR} |\zeta|^2 \psi(\zeta)^2 \dif \zeta.
\end{align}
The exponential decay of $\psi$ implies boundedness of the integrals above.

%Our proof will use rather straightforward and classical machinery from the theory of $L^p$ spaces together with some ideas borrowed from the renormalization theory in dynamics.

The operator $(\ref{operator})$ will be referred to as a {\it renormalization operator}: the equation
$$\psi_p(\eta)=\cR_\beta[\psi_p](\eta)$$
says informally that the time evolution of the initial data at a {\it later time} looked at {\it at a larger scale} is equivalent to the initial data itself.

\section{A-priori bounds renormalization fixed point}\label{secapriori}
In this section we will consider the family $\cR_\beta$ given in $(\ref{operator})$  for $\gamma=2$.

The operator $\cR_\beta$ has two ``simple'' fixed points: $\psi=0$ and $\psi=1$, corresponding to the trivial and the distributional solutions of  $(\ref{eq:modelB})$, $u(x)=0$ and $u(x)=\delta(x)$.   Additionally, $\cR_\beta$ preserve the set of even functions.

In what follows we will construct a {\it convex} $\cR_\beta$-invariant,  {\it equitight}, {\it locally equicontinous} family $\cN_p$ of {\it exponentially decaying even} functions in $L^{p}(\RR)$, $p \ge 1$. The key result that  allows us to claim precompactness of $\cN_p$, and, consequently, existence of fixed points, is the  following (see \cite{HHM}):
\begin{FKRW}
  Let $\cN_p$ be a subset of  $L^{p}(\RR)$ with $p \in [1, \infty)$, and let $\tau_\delta f$ denote the translation by $\delta$, $(\tau_\delta f)(x):=f(x-\delta)$. The subset $\cN_p$ is relatively compact iff the following properties hold:

    \vspace{1mm}
    
    \begin{itemize}
    \item[$1)$] {\it Equicontinuity}: $\lim_{|\delta| \rightarrow 0} \|\tau f - f \|_p=0$ uniformly on $\cN_p$;

      \vspace{1mm}
      
    \item[$2)$] {\it Equitightness}: $\lim_{r \rightarrow \infty} \int_{|x|>r} |f|^p =0$ uniformly on $\cN_p$.
    \end{itemize}
      
\end{FKRW}
By analogy with the renormalization theory in dynamics, existence of a renormalization-invariant precompact set will be called {\it a-priori} bounds.

We will denote for brevity:
\begin{equation} \label{J}
   J(\eta)= e^{\beta^2 \eta^2} \int_{\beta \eta}^\eta  e^{-\zeta^2} (\psi \bullet \psi)(\zeta)  \ \dif \zeta.
\end{equation}
We remark, that by Young's convolution inequality
\begin{equation}
  \label{Jbound} \left| J(\eta) \right| \le C \hspace{0.5mm}  e^{\beta^2 \eta^2} \| \psi\|_s \| \psi \|_r \left( {\rm erf}(\sqrt{q} \eta)-{\rm erf}(\sqrt{q} \beta \eta) \right)^{1\over q},
\end{equation}
whenever $\psi \in L^r(\RR) \cap L^s(\RR)$, $1/r+1/s+1/q=2$. Here, and below, $C$ serves as a stand in for a constant whose value is irrelevant to the proof. % We use $(\ref{Jbound})$ in the following
\begin{proposition}
  For all $\beta \in (0,1)$ and  for all $p>3/2$ the operator $\cR_\beta$ is a well-defined continuous operator of $L^p(\RR)$ into itself. 
\end{proposition}
\begin{proof}
  Let $s=r=p$ in $(\ref{Jbound})$, i.e. $q=p/(2 p-2)$.  Notice, that
  \begin{equation*}
    {\rm erf}(\sqrt{s} b)\hspace{-0.2mm}-\hspace{-0.2mm}{\rm erf}(\sqrt{s} a)=  \hspace{-0.5mm} \int_{\sqrt{s} a}^{\sqrt{s} b} \hspace{-2.5mm}  e^{-t^2}  \hspace{0.5mm} \dif t \le  \hspace{-0.5mm}  \int_{\sqrt{s}a}^{\sqrt{s}b}  \hspace{-2.5mm} e^{-\sqrt{s} a t}  \hspace{0.5mm} \dif t=  \hspace{-0.5mm} {1 \over \sqrt{s} a }  \hspace{-0.5mm} \left( e^{-s a^2}  \hspace{-0.5mm} -\hspace{-0.1mm} e^{-s ab} \right) \hspace{-0.6mm}.
  \end{equation*} 
Therefore, we get, using $(\ref{Jbound})$:
\begin{align}
  \nonumber \| J \|_p^p & \le  C \| \psi\|_p^{2 p} \int_\RR  \left( e^{q \beta^2 \eta^2}( {\rm erf}(\sqrt{q} \eta)-{\rm erf}(\sqrt{q} \beta \eta) )\right)^{p\over q} \ \dif \eta \\
  \nonumber & \le   C \| \psi\|_p^{2 p} \int_\RR  \left({ 1 -  e^{-q (1-\beta) \beta \eta^2}    \over \eta }\right)^{p\over q} \ \dif \eta \\
  \nonumber & \le   C \left( 1-\beta  \right)^{ p - q \over 2 q}  \| \psi\|_p^{2 p} \int_\RR  \left({ 1 -  e^{-x^2}    \over x }\right)^{p\over q} \ \dif x.
\end{align}
The last integral converges for all $p>q=p/(2 p-2)$. The claim follows.
\end{proof}
We will now construct a renormalization invariant subset in  $L^p(\RR)$.
\begin{proposition}\label{bounds}
  For any $a>0$ the subset $\cE_{a,k} \subset \cap_{p>0} L^p(\RR)$ of exponentially decaying even functions,
  \begin{equation}
\label{cEa}    \cE_{a,k}=\left\{f - {\rm measurable \ and \ even \ on } \ \RR: \ |f(x)| \le  k e^{-a |x|} \right\}, 
  \end{equation}
is invariant under $\cR_\beta$ for all $\beta \in (0,1)$.
\end{proposition}
\begin{proof}
It is straightforward to check that $\cR_\beta$ maps even functions to even functions. Set
  $$\phi(\eta)=k e^{-\eta^2} \psi(\eta).$$
The operator $\cR_\beta$ acts on $\phi$ as follows:
\begin{equation}
\label{operator3}  \tilde \cR_\beta[\phi](\eta)=\phi\left({\eta \over \beta} \right)+ 2 \  \int_{\eta }^{\eta \over \beta}    \int_0^\zeta \phi(z) \phi(\zeta-z) e^{- 2z(\zeta-z)}  d z  \ \dif \zeta.
\end{equation}
Consider an upper bound $k e^{-a \eta}$ for non-negative $\eta$:
\begin{align*}
   \tilde \cR_\beta[\phi](\eta)& \le k e^{-{\eta^2 \over \beta^2}-a {\eta \over \beta}}+ 2 k^2 \  \int_{\eta }^{\eta \over \beta}    \int_0^\zeta  e^{-z^2-a z}  e^{-(\zeta-z)^2-a (\zeta-z)} e^{-2z(\zeta-z)}  d z  \ \dif \zeta \\
   & = k  e^{-{\eta^2 \over \beta^2}-a {\eta \over \beta}}+ 2 k^2 \  \int_{\eta }^{\eta \over \beta}   e^{-\zeta^2-a \zeta} \zeta \ \dif \zeta,
 \end{align*}  
and  $ \tilde \cR_\beta[\phi](\eta) \le   k e^{-\eta^2-a\eta }$  for non-negative $\eta$ if
\begin{align*}  
   \tilde \cR_\beta[\phi](\eta) -   k e^{-\eta^2-a\eta }  &=    k \int_{\eta }^{\eta \over \beta} \hspace{-0.5mm}  (-2 \zeta -a)  e^{-{\zeta^2 }-a {\zeta}} + 2 k  \zeta   e^{-\zeta^2-a \zeta}  \  \dif \zeta
\end{align*}  
is less or equal to $0$. A sufficient condition for this to be non-positive for all $\beta \in (0,1)$ is the non-positivity of the integrand, i.e. $k \le1$.

The lower bound $f(x) \ge  -k e^{-a |x|}$ is proved in a similar way.
\end{proof}
The set $\cE_{a,k}$ contains the trivial function. We will now introduce an extra condition which will define a convex subset of $\cE_{a,k}$ that does not contain $0$.

Consider the weighted space $L^1_{u}(\RR)$ where $u$ is some weight.
\begin{lemma}
  For every $\sigma>-3$ and $\beta \in (0,1)$, the space $L^1_u(\RR)$ with $u(\eta)=|\eta|^\sigma e^{-|\eta|^2}$ is $\cR_\beta$ invariant. Additionally, for every $0<\beta_0<1$  and $-3<\sigma<-1$ there exists $\mu_0>0$,
  \begin{equation}
    \label{mu}  \mu_0 = {k^2 \over \beta_0^2} {1 - \beta_0^2 \over \beta_0^{1+\sigma}- 1}   \left( \hspace{-0.5mm}  \Gamma(s_1)   \phantom{}_1F_1 \left( \hspace{-0.5mm}s_1, {1 \over 2}, {a^2 \over 4}  \right) -  a   \Gamma(s_2)  \phantom{}_1F_1 \left( \hspace{-0.5mm}s_2, {3 \over 2}, {a^2 \over 4}  \right)  \hspace{-0.5mm} \right) \hspace{-0.5mm},
  \end{equation}
where $s_1=(\sigma+3)/2$ and $s_2=(\sigma+4)/2$, such that for all $\beta \in (\beta_0,1)$  and all $\mu >\mu_0$ the convex set $\cE_{a,k} \cap  \cM_{\mu,\sigma}$,
  \begin{align}
    \label{cMmua} \cM_{\mu,\sigma}:&=\{\psi \in L^1_u(\RR): \cI[\psi] \ge \mu \}, \hspace{1.5mm} {\rm where} \hspace{2.0mm} \\
    \nonumber \cI[\psi]:&=  \int_\RR \psi(\eta) e^{-|\eta|^2} |\eta|^\sigma  \ \dif \eta,
  \end{align}
  is $\cR_\beta$ invariant.   
\end{lemma}
\begin{proof}
  Consider an upper bound on the $L^1_u$-norm of $\cR_\beta[\psi](\eta)$:
  \begin{align}
    \nonumber \left\| \cR_\beta[\psi] \right\|_u  & \le  \int_\RR  e^{-{|\eta|^2 \over \beta^2}} \psi\left({\eta \over \beta}  \right) |\eta|^\sigma \ \dif \eta + 2 \int_\RR \left|  \int_\eta^{\eta \over \beta}  e^{-\zeta^2} (\psi \bullet \psi)(\zeta) \ \dif \zeta \right| |\eta|^\sigma \ \dif \eta \\
    \nonumber  &\le  \beta^{\sigma+1} \| \psi \|_u + 2  k^2  \int_\RR \left| \int_\eta^{\eta \over \beta} e^{-\zeta^2-a |\zeta|} \zeta \ \dif \zeta \right|  |\eta|^\sigma \ \dif \eta \\
%    \nonumber & \le  \beta^{\sigma+1} \| \psi \|_u + 4  k^2  \int_0^\infty  \left({1 \over \beta} -1  \right) \eta e^{-\eta^2-a \eta} {\eta \over \beta} \eta^\sigma \ \dif \eta \\
            \nonumber   & \le  \beta^{\sigma+1} \| \psi \|_u + {2 k^2}  \left({1 \over \beta^2} -1  \right)   \int_0^\infty  e^{-\eta^2 -a \eta}  \eta^{2+\sigma} \ \dif \eta.
  \end{align}
A sufficient condition for the last integral to converge is $\sigma >-3$.  Moreover,
  \begin{equation*}
    \cI[ \cR_\beta[\psi]]  \ge  \beta^{\sigma+1} \cI[ \psi ] - {2 A k^2 \over \beta^2} \left(1 -\beta^2 \right),
  \end{equation*}
  where
\begin{equation*}
  A={1 \over 2} \left( \Gamma \left({\sigma+3 \over 2} \right) \phantom{}_1F_1 \left({\sigma+3 \over 2}, {1 \over 2}, {a^2 \over 4}  \right)  -a  \Gamma \left({\sigma+4 \over 2} \right)  \phantom{}_1F_1 \left({\sigma+4 \over 2}, {3 \over 2}, {a^2 \over 4}  \right) \right).
  \end{equation*}
  A sufficient condition for $\cI[\cR_\beta[\psi]] \ge  \mu$ whenever  $\cI[\psi] \ge  \mu$ is $\sigma<-1$ and 
\begin{equation}
\label{mu1} \beta^{\sigma+1} \mu - {2 A k^2 \over \beta^2} \left(1 -\beta^2 \right) \ge \mu \iff \mu  \ge {2 A \over \beta^{\sigma+3}} {1-\beta^2 \over 1-\beta^{-1-\sigma}  }.
\end{equation}
Since the function $\beta^{\nu-2}(1-\beta^2)/(1-\beta^\nu)$ is decreasing on $(\beta_0,1)$ for  $0<\nu<2$, we get that $\mu \ge \mu_0$,  with $\mu_0$ as in $(\ref{mu})$, is a sufficient condition for $(\ref{mu1})$.
\end{proof}

%\medskip

  We would now like to show that  for any fixed $0<k<1$ there is a choice of $a>0$,  $0<\beta_0<1$ and $\sigma$ such that the convex sets $\cE_{a,k}$ and $\cM_{\mu,\sigma}$, defined in $(\ref{cEa})$ and $(\ref{cMmua})$,   have a non-empty intersection.

%\medskip

\begin{lemma} \label{nonempty}
   For every $-3<\sigma<-1$, $0<k<1$ and  $\beta_0 \in (0,1)$  there exists $a_0>0$ such that for any $a>a_0$ the sets $\cE_{a,k}$ and $\cM_{\mu,\sigma}$ have a non-empty intersection.
\end{lemma}
\begin{proof}
  Consider the integral $\cI[\phi]$ for $\phi(\eta)=k \min\{1,|\eta|^\nu\} e^{- a |\eta|}  \in \cE_{a,k}$, $\nu>-1-\sigma>0$:
  \begin{align}
    \nonumber \cI[\phi] &= 2 k \int_0^1  e^{- a \eta}  e^{-\eta^2} \eta^{\nu+\sigma} \ \dif \eta +  2 k \int_1^\infty  e^{- a \eta} e^{-\eta^2}  \eta^{\sigma} \ \dif \eta \\
    \nonumber  & \ge  2 k e^{-1} \int_0^1  e^{- a \eta}  \eta^{\nu+\sigma} \ \dif \eta +  2 k \int_1^a  e^{- 2 a \eta}  \eta^{\sigma} \ \dif \eta +  2 k \int_a^\infty  e^{- 2 \eta^2}  \eta^{\sigma} \ \dif \eta \\
    \label{Ipsi1}  & \ge  2 k {\gamma(\sigma\hspace{-0.2mm}  + \hspace{-0.2mm}  \nu \hspace{-0.2mm}  + \hspace{-0.2mm}  1  ,a) \over  e a^{\sigma+\nu+1} } +  { k \over 2^{\sigma} a^{1+\sigma} }  \hspace{-0.2mm} \gamma(1\hspace{-0.2mm}  +\hspace{-0.2mm}  \sigma, 2 \eta)\vert_a^{a^2}+ { k \over 2^{{\sigma +1 \over 2}} }\Gamma \left( \hspace{-0.2mm} {\sigma  \hspace{-0.2mm} +  \hspace{-0.2mm} 1\over 2},2 a^2   \hspace{-0.2mm} \right) \hspace{-0.5mm}.
    %\label{Ipsi1}  & \ge & 2 k {\gamma(\sigma\hspace{-0.5mm}  + \hspace{-0.5mm}  \nu \hspace{-0.5mm}  + \hspace{-0.5mm}  1  ,a) \over  e a^{\sigma+\nu+1} } +  { k \over 2^{\sigma} a^{1+\sigma} }  \hspace{-0.5mm}  \left(\gamma(1\hspace{-0.5mm}  +\hspace{-0.5mm}  \sigma, 2 a^2)\hspace{-0.5mm}  -\hspace{-0.5mm}  \gamma(1\hspace{-0.5mm}  +\hspace{-0.5mm}  \sigma, 2 a) \hspace{-0.5mm}   \right)\hspace{-0.5mm} + { k \over 2^{{\sigma +1 \over 2}} }\Gamma \left({\sigma +1\over 2},2 a^2  \right).
%    \label{Ipsi1}  && \phantom{ 2 k {\gamma(\sigma+\nu+1  ,a) \over a^{\sigma+\nu+1} }   }+ { k \over 2^{{\sigma +1 \over 2}} }\Gamma \left({\sigma +1\over 2},2 a^2  \right).
%    = 2 {\gamma(\nu+\sigma+1,a) \over a^{\nu+\sigma+1}} +2 {\Gamma(\sigma+1,a) \over a^{\sigma+1}}.
  \end{align}
  The asymptotics of this expression as $a \rightarrow \infty$ is
\begin{equation}
  \label{Ipsi2} C k \left(  {1  \over a^{\sigma+\nu+1} } +  {  e^{-2 a} \over 2^{\sigma} a }  \left(1 -e^{-2 a (a-1)} a^{ \sigma}  \right) + { e^{-2 a^2} a^{\sigma-1} \over 2^{{\sigma +1 \over 2}} }  \right) \asymp  {C k  \over a^{\sigma+\nu+1} }.
\end{equation}
We would like to compare $(\ref{Ipsi2})$ to the  asymptote of  $(\ref{mu})$. The confluent hypergeometric function has the following expansion for $x \rightarrow \infty$:
\begin{equation*}
  \phantom{1}_1F_1(u,v,x)={\Gamma(v) \over \Gamma(v \hspace{0.0mm} - \hspace{0.0mm}  u)} {1 \hspace{0.0mm} + \hspace{0.0mm}  O\left({1 \over x} \right) \over (-x)^{u}} \hspace{-0.0mm} + {\Gamma(v) \over \Gamma(u)  } {e^x \over  x^{v \shortminus u}} \hspace{0.0mm} \sum_{k=0}^\infty {\Gamma(v\hspace{0.0mm}-\hspace{0.0mm}u\hspace{0.0mm}+\hspace{0.0mm}k) \Gamma(1\hspace{0.0mm}-\hspace{0.0mm}u\hspace{0.0mm}+\hspace{0.0mm}k) \over \Gamma(v\hspace{0.0mm}-\hspace{0.0mm}u) \Gamma(1\hspace{0.0mm}-\hspace{-0.0mm}u) \ k!  \ x^{k}}.
\end{equation*}
Using this expansion, we obtain that the parts of the terms in  $(\ref{mu})$ proportional to ${\rm exp}(a^2/4)$ cancel, and the lower bound  $\mu_0$ becomes 
\begin{equation}
  \label{mu2} \mu_0 \asymp {k^2 \over a^{\sigma+3}}.
\end{equation}
Therefore, if $\sigma+3>\nu+\sigma+1 \iff \nu<2$ then for any fixed $k$ the bound $(\ref{mu2})$ decreases faster as $a \rightarrow \infty$ than the lower bound  $(\ref{Ipsi2})$ on $\cI[\psi]$, and the conclusion follows.
\end{proof}
We continue with the second and the third ingredient of our {\it a-priori} bounds: {\it equitightness} and {\it equicontinuity} in the sense of $L^p(\RR)$ spaces.

Since $\cE_{a,k} \subset L^p(\RR)$ for all $p > 0$, we see that the exponentially decaying  family $\cE_{a,k}$ is {\it equitight}:
\begin{corollary} \label{equitight}
  The set $\cE_{a,k}$ is equitight for all $p >0$, i. e., for all $f \in \cE_{a,k}$
  \begin{equation}
\lim_{r \rightarrow \infty} \int_{|x|>r} |f|^p=0
  \end{equation}
uniformly on $\cE_{a,k}$.
\end{corollary}
Let $\cE_{a,k,K,\alpha}^p$ be the subset of all {\it $\alpha$-H\"older uniformly continuous} functions with constant $K$ in $\cE_{a,k}$, i.e. all functions $f \in \cE_{a,k}$ such that for every such $f$ there exists a non-negative valued $\omega_\psi \in L^p(\RR)$, such that  
\begin{equation}\label{Holder}
| \tau_\delta f(\eta)-f(\eta) | \le \omega_f(\eta) |\delta|^\alpha,  \hspace{3mm} \| \omega_f\|_p \le K, 
\end{equation}
where $(\tau_\delta f)(x)=f(x-\delta)$. We will  also denote the subset of  {\it $\alpha$-H\"older locally uniformly continuous} functions as $\cE_{a,k,K,\alpha,\delta_0}^p$, i.e. the subset for which $(\ref{Holder})$ holds for all $|\delta| < \delta_0$. Both of these families are {\it equicontinous} in $L^p(\RR)$, i.e.
\begin{equation}
\label{equicont}  \lim_{|\delta| \rightarrow 0}\| \tau_\delta f - f \|_p=0 
\end{equation}
uniformly on  both $\cE_{a,k,K,\alpha}^p$ and $\cE_{a,k,K,\alpha,\delta_0}^p$. By Frechet-Kolmogorov-Riesz-Weil Compactness Theorem (see \cite{HHM}),  both $\cE_{a,k,K,\alpha}^p$ and $\cE_{a,k,K,\alpha,\delta_0}^p$ are relatively compact in  $L^p(\RR)$ for all $p \ge 1$.

We will require a straightforward technical lemma before we proceed to the {\it a-priori} bounds.
\begin{lemma} \label{exp_diff}
  Fix $\beta_0 \in (0,1)$. There exist a constant $C=C(\beta_0)>0$ such that for all $\beta \in (\beta_0,1)$
\begin{equation}
  \label{L5}   E(\eta,\delta,\beta):=\left| e^{ -(1-\beta^2) {|\eta|^2 \over \beta^2}} -  e^{ -(1-\beta^2) {|\eta-\delta|^2 \over \beta^2}} \right| \le C \ \left({1 \over \beta^2}-1 \right)^{1 \over 2} |\delta|.
\end{equation}
\end{lemma}
\begin{proof}
Using the shorthanded notation $E$ for the difference of exponentials as in $(\ref{L5})$, we have in the case $|\eta|>|\eta-\delta|$ and $|\eta|>|\delta|$:
%\begin{eqnarray}
%  \nonumber  F(\eta,\delta) & <  & e^{ -{1-\beta^2 \over \beta^2} {|\eta-\delta|^2}}  \left| 1 -  e^{ {1-\beta^2 \over \beta^2} \left( |\eta-\delta|^2 - |\eta|^2    \right)} \right| \\
 % \nonumber & < &   e^{ -{1-\beta^2 \over  \beta^2}  (|\eta|-|\delta|)^2}  \left( 1 -  e^{ -{ 1-\beta^2 \over \beta^2 } D  |\eta| |\delta|} \right) \\
%\label{bound18} & < &   e^{ {1-\beta^2 \over  \beta^2}  |\delta|^2}   e^{ -{1 \over 2} {1-\beta^2 \over  \beta^2}  |\eta|^2}  \left( 1 -  e^{ -{ 1-\beta^2 \over \beta^2 } D  |\eta| |\delta|} \right)
%\end{eqnarray}
\begin{equation}
  \label{bound18}  E \hspace{-0.5mm}  \le  \hspace{-0.5mm} e^{\shortminus{1 - \beta^2 \over \beta^2} {|\eta\shortminus\delta|^2}} \hspace{-0.3mm} \left| \hspace{-0.2mm} 1\hspace{-0.5mm}  \shortminus \hspace{-0.5mm}  e^{ {1 - \beta^2 \over \beta^2} \left( |\eta\shortminus\delta|^2 \shortminus |\eta|^2    \right)} \hspace{-0.2mm} \right| \hspace{-0.5mm} \le \hspace{-0.5mm}  e^{ {1 - \beta^2 \over  \beta^2} \left(|\delta|^2\shortminus {|\eta|^2\over 2} \right)}\hspace{-1.2mm}  \left(\hspace{-0.5mm} 1 \hspace{-0.5mm} \shortminus \hspace{-0.5mm}  e^{ \shortminus{ 1-\beta^2 \over \beta^2 } D  |\eta| |\delta|}\hspace{-0.2mm} \right)
\end{equation}
for some constant $D>0$. To get a bound on this expression, proportional to a power of $|\delta|$, we first notice that for non-negative $x$,
\begin{equation}
  \label{bound9} e^{-x^2}  \left(1 - e^{-a x} \right)\le{ 2 a x \over (1 + a x) (1 +  x^2 )  } \le { 2 a x \over  (1 + x^2 )  } \le a.
\end{equation}
We can now use this estimate on $(\ref{bound18})$ with
$$x={1 \over \sqrt{2} } \left({1 \over \beta^2}-1\right)^{1\over 2} |\eta|, \quad a= \sqrt{2} D |\delta|  \left({1 \over \beta^2}-1 \right)^{1\over 2},$$
to get
\begin{equation}
  \label{L1}   E  \le  C \ e^{ {1-\beta^2 \over  \beta^2} |\delta|^2}    \left({1 \over \beta^2} - 1 \right)^{1\over 2}   \ |\delta|.
\end{equation}
In the case $|\eta|>|\eta-\delta|$ and $|\eta|<|\delta|$,
\begin{equation}
  \label{L2}   E  \le      1 -  e^{ -{ 1-\beta^2 \over \beta^2 } D |\delta| |\eta|} \le  C   \left({1 \over \beta^2} - 1 \right)  |\delta|^2 .
\end{equation}
In the case $|\eta-\delta|>|\eta|>|\delta|$, we can again use estimate $(\ref{bound9})$:
\begin{equation}
  \label{L3}   E  \le   e^{ -{1-\beta^2 \over \beta^2} {|\eta|^2 }} \left(1-e^{-{1-\beta^2 \over \beta^2} D|\delta| |\eta|} \right)  \le  C\  \left({1 \over \beta^2}-1 \right)^{1\over 2}   \ |\delta|.
\end{equation}
%\begin{equation}
%  \label{L3}   \left| e^{ -(1-\beta^2) {|\eta|^2 \over \beta^2}} -  e^{ -(1-\beta^2) {|\eta-\delta|^2 \over \beta^2}} \right| <  C\  \left({1 \over \beta^2}-1 \right)^{1\over 2}   \ |\delta|.
%\end{equation}
In the case $|\eta-\delta|>|\eta|$ and $|\delta|>|\eta|$, we can crudely bound as follows:
\begin{equation}
\label{L4}   E   \hspace{-0.5mm}  \le  \hspace{-0.5mm}   e^{ -{1-\beta^2 \over \beta^2} {|\eta|^2 }}  \hspace{-0.5mm}  \left(1  \hspace{-0.5mm}  -  \hspace{-0.5mm}   e^{-{1-\beta^2 \over \beta^2} D|\eta| |\delta|} \right)  \hspace{-0.5mm}   \le   \hspace{-0.5mm}   1  \hspace{-0.5mm}   -   \hspace{-0.5mm}   e^{-{1-\beta^2 \over \beta^2} D|\delta|^{2}}  \le  C \ \left({1 \over \beta^2} -1\right)  |\delta|^{2} \hspace{-0.5mm} .
\end{equation}
The claim follows after one combines bounds $(\ref{L1})$, $(\ref{L2})$,  $(\ref{L3})$ and  $(\ref{L4})$.
\end{proof}

%\medskip

We continue with the proof of the {\it a-priori} bounds. Per our convention, $C$ denotes a constant whose size has no bearing on the arguments in the proofs (they may, however depend on $\beta_0$, $\delta_0$ and $a_0$). %, specifically, they are decreasing functions of $a_0$).

%Throughout the rest of the paper we will denote constants independent of $A$, $k$, $\alpha$, $K$, $\beta \in (\beta_0,1)$ and $a>a_0$  and  whose size has no bearing on the argument, by a generic stand-in $C$ (this constants will, however depend on $\beta_0$, $\delta_0$ and $a_0$, specifically, they are decreasing functions of $a_0$).

%\medskip

\begin{proposition}\label{apriori2}(\underline{A-priori bounds: equicontinuity.}) 
  For every $p \ge 1$ there are constants $0<\beta_0 <1$, $A>0$, $K>0$, $a>0$, $\delta_0>0$,  $k>0$ and $0<\alpha < {1 / p}$, such that for every $\beta \in (\beta_0,1)$ the convex subset
\begin{equation}
\label{cEpak}  \cE^{p}_{a,k,K,A,\alpha, \delta_0}=\{\psi \in \cE^{p}_{a,k,K,\alpha, \delta_0}: \omega_\psi(\eta) \le A |\eta| e^{-a |\eta|} \hspace{1.5mm} \mathrm{for} \hspace{1.5mm} |\eta| > 1\}
\end{equation}
of  $\cE^{p}_{a,k,K,\alpha, \delta_0}$, is renormalization invariant under $\cR_\beta$.
\end{proposition}
\begin{proof}
  Assume that $\psi \in  \cE^p_{a,k,K,A,\alpha, \delta_0}$. We consider  linear and non-linear terms separately for $|\delta|<\beta \delta_0$. Clearly,  it is sufficient to prove the result for $0 \le  \delta < \beta \delta_0$.

  \noindent \textit{\textbf{1)  Linear terms, case $\bm{0 < \delta<  \beta \delta_0}$.}}   Denote,
  \begin{equation}
\label{L}    L(\eta,\delta,\beta)=  \left| e^{ -(1-\beta^2) {|\eta|^2 \over \beta^2}} \psi\left( {\eta \over \beta} \right) -  e^{ -(1-\beta^2) {|\eta-\delta|^2 \over \beta^2}} \psi\left( {\eta -\delta \over \beta} \right) \right|.
  \end{equation}
  We have for all $|\delta| < \beta \delta_0$, using the notation of $(\ref{L5})$
  \begin{align}
%    \nonumber   L(\eta,\delta) \hspace{-1.2mm} &  \hspace{-1.2mm}  \le   \hspace{-1.2mm}  &   \hspace{-1.2mm}  \left| e^{ -(1-\beta^2) {|\eta|^2 \over \beta^2}} \psi\left( {\eta \over \beta} \right) -  e^{ -(1-\beta^2) {|\eta-\delta|^2 \over \beta^2}} \psi\left( {\eta -\delta \over \beta} \right) \right|  \\
    \nonumber    L(\eta,\delta,\beta) & \le   E(\eta,\delta,\beta) \left| \psi  \hspace{-0.5mm}   \left( {\eta -\delta  \over \beta} \right) \right| +  e^{ -{1-\beta^2 \over \beta^2} {|\eta|^2}} \left| \psi  \left( {\eta \over \beta} \right)  -  \psi  \left( {\eta -\delta \over \beta} \right) \right|  \\
    \nonumber  & \le  E(\eta,\delta,\beta)     \left| \psi\left( {\eta -\delta \over \beta} \right) \right|   +  \beta^{-\alpha}  e^{ -{1-\beta^2 \over \beta^2} {|\eta|^2}}   \omega_\psi \left({\eta \over \beta} \right) |\delta|^\alpha \\
 \nonumber  & =  L_1(\eta,\delta,\beta)+L_2(\eta,\delta,\beta).
  \end{align}
We have 
\begin{equation*}
E(\eta,\delta,\beta) \le  \left\{ 1- e^{ -{ 1-\beta^2 \over \beta^2} \left( |\eta|^2 -|\eta-\delta|^2\right) } , \hspace{2mm}  |\eta| > |\eta-\delta|, \atop  1- e^{ -{ 1-\beta^2 \over \beta^2} \left( |\eta-\delta|^2 -|\eta|^2 \right) }, \hspace{2mm}  |\eta| \le |\eta-\delta|.    \right.
\end{equation*}
  %Suppose $|\eta|>|\eta-\delta|$ then 
%\begin{equation}
%  \nonumber   \left| e^{ -(1-\beta^2) {|\eta|^2 \over \beta^2}} -  e^{ -(1-\beta^2) {|\eta-\delta|^2 \over \beta^2}} \right| \le   \left|1- e^{ -{ 1-\beta^2 \over \beta^2} \left( |\eta|^2 -|\eta-\delta|^2\right) } \right|.
%\end{equation}
%Otherwise, if $|\eta| \le |\eta-\delta|$ then 
%\begin{equation}
%  \nonumber   \left| e^{ -(1-\beta^2) {|\eta|^2 \over \beta^2}} -  e^{ -(1-\beta^2) {|\eta-\delta|^2 \over \beta^2}} \right| \le  \left|1- e^{ -{ 1-\beta^2 \over \beta^2} \left( |\eta-\delta|^2 -|\eta|^2 \right) } \right|.
%\end{equation}
In both cases, if $|\eta| > 2 |\delta|$, the following very conservative estimate holds:
\begin{equation*}
  E(\eta,\delta,\beta) \le 1- e^{ -{ 1-\beta^2 \over \beta^2} D |\eta| |\delta| } \le  { 1-\beta^2 \over \beta^2} D |\eta| |\delta|
\end{equation*}
for some constant $D>0$. Therefore,
\begin{align*}
   \| L_1 \|_{p}^p & \le   \hspace{-1mm}   \int \displaylimits_{|\eta| \le 2 |\delta|  }  \hspace{-2mm} E(\eta,\delta,\beta)^p  \left| \psi\left( {\eta -\delta  \over \beta} \right) \right|^p  \dif \eta  + \hspace{-1.5mm}   \int \displaylimits_{|\eta| > 2 |\delta|  }  \hspace{-2mm} E(\eta,\delta,\beta)^p  \left| \psi\left( {\eta -\delta \over \beta} \right) \right|^p  \dif \eta \\
    & \le  k \left( \hspace{-0.5mm}   1\hspace{-0.5mm}  -\hspace{-0.5mm}  e^{\shortminus {1-\beta^2  \over \beta^2 } |3 \delta|^2   } \hspace{-0.5mm}     \right)^p  \hspace{-3mm} \int \displaylimits_{|\eta| \le 2 |\delta|  }  \hspace{-3mm}  e^{\shortminus p a {|\eta-\delta| \over \beta}}    d \eta + C k (1\hspace{-0.5mm} -\hspace{-0.5mm} \beta)^p |\delta|^p \hspace{-3.5mm} \int \displaylimits_{|\eta| > 2 |\delta|  }  \hspace{-3mm} |\eta|^p e^{\shortminus p a {|\eta-\delta| \over \beta}}  \dif \eta \\
    & \le  C k (1-\beta)^p \left( |\delta|^{3 p} +|\delta|^p \right).
\end{align*}
We have, therefore, 
\begin{equation} \label{Lestimate2}
  \| L \|_{p}   \le   C k  (1-\beta)  |\delta| + \beta^{{1 \over p}-\alpha}  K  |\delta|^\alpha.
  \end{equation}

Additionally, in the case $\eta>1$ and $\delta>0$, we have the following bound on the linear terms themselves (and not the norm):% for $|\eta|>2 \delta$:
\begin{eqnarray}
  \nonumber L(\eta,\delta)  &\le& C k (1-\beta)|\eta|  \delta e^{ -a{\eta-\delta \over \beta}} + \beta^{-\alpha} e^{-{1-\beta^2 \over \beta^2} |\eta|^2  } \omega_{\psi}\left({\eta \over \beta}\right) \delta^\alpha \\
   \label{Lterms}  &\le& C k (1-\beta)|\eta|  \delta e^{ -a{\eta \over \beta}} + A \beta^{-\alpha-1} e^{-{1-\beta^2 \over \beta^2} |\eta|^2  }|\eta|  e^{-a {|\eta| \over \beta}   } \delta^\alpha
\end{eqnarray}

%\bigskip

\noindent \textit{\textbf{2)  Non-linear terms, case $\bm{0 <\delta<  \beta \delta_0}$.}}  We consider $\|\hat J -\tau_\delta \hat J\|_p$, $\hat J= 2 J \circ \beta^{-1}$ where $J$ is defined in  $(\ref{J})$.  We have
$$\hat J=   2 \eta  \int_{1}^{1 \over \beta}  e^{|\eta|^2 (1-t^2)} \int_{0}^{\eta t}   \psi(x) \psi(\eta t  -  x)  \ \dif x    \ \dif  t,$$
and
 $$|\hat J(\eta)-\hat J(\eta-\delta)| \le J_1(\eta,\delta)+J_2(\eta,\delta)+J_3(\eta,\delta) + J_4(\eta,\delta),$$
 where
 \begin{align}
   \nonumber J_1(\eta,\delta)  & \le    C |\eta|  \int_{1}^{1 \over \beta}   \left| e^{|\eta|^2 \left(1-t^2\right)}  - e^{|\eta-\delta|^2 \left(1-t^2\right)}\right| \left|  \int_{0}^{\eta t}   |\psi(x)| |\psi(\eta t  -  x)|   \ \dif x  \right| \  \dif  t, \\
   \nonumber J_2(\eta,\delta)  & \le  C  \delta \int_{1}^{1 \over \beta}   e^{|\eta-\delta|^2 \left(1-t^2\right)}  \left|  \int_{0}^{\eta t} |\psi(x)| |\psi(\eta t-x)|    \  \dif x    \  \dif  t \right|, \\
   \nonumber J_3(\eta,\delta)  &\le C  |\eta \hspace{-0.3mm} \shortminus \hspace{-0.3mm}  \delta|   \int_{1}^{1 \over \beta} \hspace{-1mm}  {e^{|\eta\shortminus \delta|^2 \left(1\shortminus t^2\right)} }  \left|  \int_0^{\eta t}  \hspace{-2mm} |\psi(x)| \left| \psi(\eta t \hspace{-0.3mm} \shortminus \hspace{-0.3mm}  x)\hspace{-0.3mm} \shortminus \hspace{-0.3mm} \psi((\eta\hspace{-0.3mm} \shortminus \hspace{-0.3mm}  \delta) t \hspace{-0.3mm} \shortminus \hspace{-0.3mm}  x) \right|   \  \dif x     \  \dif  t \right|, \\
   \nonumber J_4(\eta,\delta)  &\le C  |\eta- \delta|   \int_{1}^{1 \over \beta}   {e^{|\eta-\delta|^2 \left(1-t^2\right)} }  \left| \hspace{0.6mm}  \int_{(\eta-\delta ) t}^{\eta t}  \hspace{-1mm} |\psi(x)||\psi((\eta-\delta) t -x)|   \  \dif x  \right|  \  \dif  t.
 \end{align}

 We consider each $J_i$ separately.

\vspace{3mm}
 
 \noindent \underline{\it $a)$ Calculation of $J_1$.}  Since $|(\psi \bullet \psi)(x)| \le e^{-a |x|} |x|$, we have 
 \begin{eqnarray}
   \nonumber J_1(\eta,\delta)  & \le &  C k^2  |\eta| \int_{1}^{1 \over \beta}  { \left| e^{|\eta|^2 \left(1-t^2\right)}  -  e^{|\eta-\delta|^2 \left(1-t^2\right)}\right| }  e^{-a |\eta t|} |\eta t|  \  \dif  t.
 \end{eqnarray}
 Therefore, in the case $\eta>1$, we have  $\eta \ge \eta-\delta$ and $\eta -\delta > (1-\delta) \eta$, and
\begin{align}
  \nonumber  J_1(\eta,\delta)    & \le C k^2  \int_{1}^{1 \over \beta} e^{-(\eta-\delta)^2 (t^2-1)} \left(1-e^{-\delta (2 \eta -\delta) (t^2-1)} \right)  e^{-a \eta t} \eta t   \  \dif \eta t  \\
    \nonumber   & \le C k^2  \int_{1}^{1 \over \beta} e^{-(1-\delta)^2\eta^2 (t^2-1)}  \left(1-e^{-C (1-\beta^2) \eta \delta} \right)   e^{-a \eta t} \eta t   \  \dif \eta t  \\
  \nonumber &\le C k^2 \delta \eta (1-\beta^2) e^{(1-\delta)^2\eta^2} \int_{(1-\delta) \eta}^{(1-\delta)\eta \over \beta} e^{-y^2}  e^{-{a \over 1-\delta} y} y \ d y \\
%  \nonumber  & \le      C k^2 \delta |\eta|(1-\beta^2) e^{(1-\delta)^2 \eta^2} e^{-\left( \eta+{a \over 2} \right)^2} \left(1-e^{-\left({1 \over \beta^2}-1\right) \eta^2 -\left({1 \over \beta} -1  \right) \eta a}     \right)  \\
  \label{J1terms}  &\le  C k^2 \delta \eta (1-\beta^2) e^{-a \eta} \left(1-e^{-\left({1 \over \beta^2}-1\right) \eta^2 -\left({1 \over \beta} -1  \right)a \eta }     \right),
\end{align}
while the norm itself can be bounded,  using the bound $(\ref{L5})$ on the difference of exponentials, as
\begin{equation}
  \label{J1}  \| J_1\|_{p}   \le C k^2 (1 -\beta )^{3 \over 2}    \delta.  
\end{equation}
\noindent \underline{\it $b)$ Calculation of $J_2$.} For $\eta>1$ this bound follows closely that on $J_1$:
%\begin{eqnarray}
%  \nonumber J_2(\eta,\delta)   &\le&  C  k^2 \delta \int_{1}^{1 \over \beta}  { e^{|\eta-\delta|^2 \left(1-t^2\right)}}  e^{-a |\eta t|} |\eta t|  \  \dif  t  \le C k^2 \delta  \int_{|\eta|}^{|\eta| \over \beta} e^{-a y}  \  \dif  y  \\
%  \label{J2terms}  &\le& C k^2 \delta  e^{-a |\eta|} \left(1-e^{-a\left({1\over \beta} -1 \right)|\eta|} \right),
% \end{eqnarray}
\begin{align}
  \nonumber J_2(\eta,\delta)   &\le  C  k^2 \delta \int_{1}^{1 \over \beta}  { e^{(\eta-\delta)^2 \left(1-t^2\right)}}  e^{-a \eta t} \eta t  \  \dif  t  \\
  \label{J2terms}  &\le  C k^2 \delta \eta^{-1}  e^{-a \eta} \left(1-e^{-\left({1 \over \beta^2}-1\right) \eta^2 -\left({1 \over \beta} -1  \right) a \eta}     \right),
 \end{align}
while the norm itself is bounded as
 \begin{equation}
   \label{J2} \|J_2\|_{p} \le  C k^2 \delta  \left( \int_{\RR} e^{-a p \eta} \left( {1 \over \beta} -1 \right)^p \eta^p \ d \eta  \right)^{1 \over p} \le   C k^2 (1-\beta)   \delta.
 \end{equation}

%\bigskip
  
\noindent  \underline{\it $c)$ Calculation of $J_3$.}  We first consider the case of  $|\eta| \le 1$. 
\begin{align}
  \nonumber J_3(\eta,\delta)   &   \le   C \eta \hspace{-1mm} \int_{1}^{1 \over \beta}  \hspace{-0.8mm}  {e^{(\eta-\delta)^2 \left(1-t^2 \right)}}  \hspace{-1.0mm}  \int_{0}^{\eta t}  \hspace{-1mm} |\psi(x)|  \omega_\psi(\eta t -x ) (\delta t)^\alpha    \  \dif x   \  \dif  t \\
  \nonumber   &  \le    C    { \delta^\alpha} \eta  {e^{(\eta \shortminus \delta)^2}} \hspace{-0.5mm} \int_{1}^{1 \over \beta}  \hspace{-0.5mm}  {e^{- {(\eta - \delta)^2} t^2}}   t^\alpha \ d t  \  \| \psi\|_q \ \| \omega_\psi \|_p\\
  \label{J3terms1}    &  \le   C  K k (1-\beta)  \delta^\alpha \eta,
\end{align}
by Young's inequality with $1/p+1/q+1/r=2$  and  $r=1$. Next, for  $\eta>1$,
\begin{align}
  \nonumber  \int_{0}^{\eta t}  \hspace{-1mm} |\psi(\eta t -x )|  \omega_\psi(x )   \  \dif x  &\le  \int_{0}^{1}  \hspace{-1mm} |\psi(\eta t -x)|  \omega_\psi(x )    \  \dif x  + \hspace{-1.0mm}   \int_{1}^{\eta t}   \hspace{-2mm} |\psi(\eta t - x)| \  \omega_\psi(x)   \  \dif x \\
 \nonumber  &\le   k  \int_{0}^{1}  \hspace{-1mm} e^{-a|\eta t -x|} \omega_\psi(x )    \  \dif x+  \hspace{-0.5mm}  A k  \int_{1}^{\eta t}   \hspace{-2.5mm} e^{-a (\eta t -x)}  e^{-a x} x   \  \dif x  \\
% \nonumber  &\le&   k  e^{-a |\eta t-1|} \left(1- e^{-a} \right)^{1 \over q} \|\omega_\psi \|_p   +  A k e^{-a |\eta| t} |\eta t-1|  \\
 \nonumber  &\le   C K k  e^{-a \eta t}  +  C A k e^{-a \eta t} \eta^2 t^2.
\end{align}
Therefore, in the case $\eta>1$,
\begin{align}
%  \nonumber  J_3(\eta,\delta) &\le  C K k \delta^\alpha\int_1^{1 \over \beta}  e^{-a |\eta| t}    \  \dif  |\eta| t +  C A k e^{|\eta-\delta|^2}  \delta^\alpha \int_1^{1 \over \beta} e^{-|\eta-\delta|^2 t^2 - a |\eta| t} |\eta|^2 t^2    \  \dif |\eta| t   \\
  \nonumber  J_3(\eta,\delta) &\le  C k e^{(\eta-\delta)^2}  \delta^\alpha \int_1^{1 \over \beta} e^{-(\eta-\delta)^2 t^2 - a \eta t} \left( K+A \eta^2 t^2   \right)   \  \dif \eta t   \\
  \nonumber & \le  C K k \delta^\alpha e^{\hspace{-0.3mm} \shortminus a \eta} \hspace{-0.6mm}   \left(  \hspace{-0.6mm}  1\hspace{-0.5mm} \shortminus \hspace{-0.5mm}e^{\hspace{-0.3mm} \shortminus \hspace{-0.3mm}\left({1 \over \beta} \shortminus 1 \right)a \eta} \right) \hspace{-0.6mm} +  \hspace{-0.1mm}   C A k e^{(1 \shortminus \delta)^2 \eta^2}  \delta^\alpha \eta \hspace{-3.5mm}  \int \displaylimits_{(1 \shortminus \delta) \eta}^{(1 \shortminus \delta) \eta \over \beta} \hspace{-3.5mm} e^{\hspace{-0.3mm} \shortminus \hspace{-0.3mm}\left(y+ {a \over 2(1\hspace{-0.3mm} \shortminus \hspace{-0.3mm}\delta)} \right)^2}  \hspace{-1.5mm}  y   \dif y \\
%  \nonumber & \le  C k \delta^\alpha \left( K  e^{\hspace{-0.3mm} \shortminus \hspace{-0.3mm}a|\eta|} \left(1\hspace{-0.3mm} \shortminus \hspace{-0.3mm}e^{\hspace{-0.3mm} \shortminus \hspace{-0.3mm}\left({1 \over \beta}\hspace{-0.3mm} \shortminus \hspace{-0.3mm}1 \right)a |\eta|} \right) +  \right. \\
  % \nonumber  &\hspace{20mm}+ \left.  A |\eta|  e^{|\eta\hspace{-0.3mm} \shortminus \hspace{-0.3mm}\delta|^2} e^{\hspace{-0.3mm} \shortminus \hspace{-0.3mm}\left(|\eta\hspace{-0.3mm} \shortminus \hspace{-0.3mm}\delta|+ {a \over 2} \right)^2}\left(1\hspace{-0.3mm} \shortminus \hspace{-0.3mm}e^{\hspace{-0.3mm} \shortminus \hspace{-0.3mm}\left({1 \over \beta^2}\hspace{-0.3mm} \shortminus \hspace{-0.3mm}1\right)|\eta\hspace{-0.3mm} \shortminus \hspace{-0.3mm}\delta|^2\hspace{-0.3mm} \shortminus \hspace{-0.3mm}\left({1\over\beta}\hspace{-0.3mm} \shortminus \hspace{-0.3mm}1 \right)a |\eta\hspace{-0.3mm} \shortminus \hspace{-0.3mm}\delta|} \right) \right) \\
        \label{J3terms2} & \le  C k \delta^\alpha e^{-a \eta} \hspace{-0.7mm} \left( \hspace{-0.8mm}K \hspace{-0.8mm}  \left(\hspace{-0.5mm} 1\hspace{-0.5mm} \shortminus \hspace{-0.5mm} e^{-\left({1 \over \beta}-1 \right)a \eta} \right) \hspace{-0.6mm} + \hspace{-0.6mm} A \eta \hspace{-0.6mm} \left( \hspace{-0.6mm} 1\hspace{-0.5mm} \shortminus \hspace{-0.5mm} e^{-\left({1 \over \beta^2}-1\right) \eta^2-\left({1\over\beta}-1 \right)a \eta}  \hspace{-0.3mm}  \right) \hspace{-1.5mm} \right)  \hspace{-0.5mm} .
\end{align}
\noindent  \underline{\it $d)$ Calculation of $J_4$.}  For $\eta > 1$,
 \begin{align}
   \nonumber J_4(\eta,\delta)  &\le C k^2  (\eta- \delta)   \int_{1}^{1 \over \beta}   {e^{(\eta-\delta)^2 \left(1-t^2\right)} }  \int_{(\eta-\delta ) t}^{\eta t}  \hspace{-1mm} e^{-a|x|} e^{-a|(\eta-\delta) t -x|  }  \  \dif x   \  \dif  t \\
   \nonumber   &\le C  (\eta- \delta)  k^2  \int_{1}^{1 \over \beta}   e^{(\eta-\delta)^2 \left(1-t^2\right)}  e^{-a(\eta-\delta) t} \delta t  \  \dif  t \\
    \nonumber   &\le C k^2 (\eta-\delta)^{-1}  e^{-a (\eta-\delta)} \left(1-e^{-\left({1\over \beta^2} -1 \right) (\eta-\delta)^2 -\left({1\over \beta} -1 \right) a (\eta-\delta) } \right) \delta, \\
       \label{J4terms}   &\le C k^2  \eta^{-1} e^{-a\eta} \left(1-e^{-\left({1\over \beta^2} -1 \right)  \eta^2 -\left({1\over \beta} -1 \right) a \eta} \right) \delta,
 \end{align}
 and
 \begin{equation}
       \label{J4} \| J_4 \|_p   \le C k^2 (1-\beta) \delta.
 \end{equation}

\noindent \textit{\textbf{3)  Bound on the constant $A$.}} Denote $f(\eta)= {\left({1 / \beta^2}-1\right) \eta^2+\left({1/\beta}-1 \right)a \eta}$ and $g(\eta)=\left({1/\beta}-1 \right)a \eta$. We can now combine $(\ref{Lestimate2})$, $(\ref{J1terms})$, $(\ref{J2terms})$, $(\ref{J3terms1})$, $(\ref{J3terms2})$  and $(\ref{J4terms})$ to get the following bound on $\omega_{\cR_\beta[\psi]}$ for $\eta>1$:
%\begin{align}
%  \nonumber \omega_{\cR_\beta[\psi]} &\le A {e^{-{1-\beta^2 \over \beta^2} \eta^2  } \over \beta^{\alpha+1}} \eta e^{-a {\eta \over \beta}   } +  C k (1-\beta)  \delta_0^{1-\alpha} e^{ -a {\eta \over \beta}}  + \\
%  \nonumber & \phantom{  A {e^{-{1-\beta^2 \over \beta^2} \eta^2  } \over \beta^{\alpha+1}} \eta e^{-a {\eta \over \beta}   }} \hspace{5mm}  + C k^2 (1-\beta) \delta_0^{1-\alpha} \eta  e^{-a \eta} \left(1-e^{-f(\eta)} \right) + \\
%  \nonumber  & \phantom{ A {e^{-{1-\beta^2 \over \beta^2} \eta^2  } \over \beta^{\alpha+1}} \eta e^{-a {\eta \over \beta}   }} \hspace{5mm} + C k^2 \delta_0^{1-\alpha} |\eta|^{-1}  e^{-a |\eta|} \left(1-e^{-f(\eta)}     \right)+ \\
%   \label{wrpsi} & \phantom{ A {e^{-{1-\beta^2 \over \beta^2} \eta^2  } \over \beta^{1+\alpha}} \eta  e^{-a {\eta \over \beta}   }} \hspace{5mm} +  C k e^{-a \eta} \left(K (1-\beta)   +  A  \eta \left(1-e^{-f(\eta)} \right)  \right).
% %  & \phantom{ A {e^{-{1-\beta^2 \over \beta^2} \eta^2  } \over \beta^{1+\alpha}} \eta  e^{-a {\eta \over \beta}   }}  \hspace{5mm} + C k^2 \delta_0^{1-\alpha} e^{-a\eta} \left(1-e^{-a \left({1\over \beta} -1 \right) \eta} \right).
%\end{align}
\begin{align}
  \nonumber \omega_{\cR_\beta[\psi]} &\le A {e^{-{1-\beta^2 \over \beta^2} \eta^2  } \over \beta^{\alpha+1}} \eta e^{-a {\eta \over \beta}   } +  C k (1-\beta)  \delta_0^{1-\alpha} e^{ -a {\eta \over \beta}}  + \\
  \nonumber & \phantom{  A {e^{-{1-\beta^2 \over \beta^2} \eta^2  } \over \beta^{\alpha+1}} \eta e^{-a {\eta \over \beta}   }} \hspace{5mm}  + C k^2 (1-\beta) \delta_0^{1-\alpha} \eta  e^{-a \eta} \left(1-e^{-f(\eta)} \right) + \\
  \nonumber  & \phantom{ A {e^{-{1-\beta^2 \over \beta^2} \eta^2  } \over \beta^{\alpha+1}} \eta e^{-a {\eta \over \beta}   }} \hspace{5mm} + C k^2 \delta_0^{1-\alpha} |\eta|^{-1}  e^{-a |\eta|} \left(1-e^{-f(\eta)}     \right)+ \\
   \label{wrpsi} & \phantom{ A {e^{-{1-\beta^2 \over \beta^2} \eta^2  } \over \beta^{1+\alpha}} \eta  e^{-a {\eta \over \beta}   }} \hspace{5mm} +  C k e^{-a \eta} \left(K (1-\beta)   +  A  \eta \left(1-e^{-f(\eta)} \right)  \right).
\end{align}
A sufficient condition for  $\omega_{\cR_\beta[\psi]} \le A \eta e^{-a \eta}$ is therefore
%\begin{align}
%  \nonumber A  &   \ge   A {e^{ -f(\eta)} \over \beta^{\alpha+1}}  + C k (1-\beta)  \delta_0^{1-\alpha} e^{-g(\eta)}  + C k^2 (1-\beta) \delta_0^{1-\alpha} \left(1-e^{-f(\eta) } \right)   + \\
%  \nonumber  & \phantom{ A {e^{ -f(\eta)} \over \beta^{\alpha+1}}} \hspace{5mm} + C k^2 \delta_0^{1-\alpha} \left(1-e^{-f(\eta)} \right) +  C K k (1-\beta)   + C k  A  \left(1-e^{-f(\eta)} \right),
%\end{align}
\begin{equation}
  \nonumber A     \ge   A {e^{ -f(\eta)} \over \beta^{\alpha+1}}  + C k (1-\beta)  \delta_0^{1-\alpha} e^{-g(\eta)}   + C k(k \delta_0^{1-\alpha} +A) \left(1-e^{-f(\eta)} \right) +  C K k (1-\beta),
\end{equation}
which is implied by
\begin{equation}
  \nonumber A {\beta^{1+\alpha} - e^{-f(1)} \over 1 - e^{-f(1)} } -A C k \ge C k  \delta_0^{1-\alpha} e^{-g(1)}  + C k^2 \delta_0^{1-\alpha} + C K k.
\end{equation}
As $\beta \rightarrow 1$, this condition is tangent to 
\begin{equation}
  \label{Acond} A \left( 1- C k -{1 +\alpha \over 2 +a} \right)\ge C k \left( e^{-g(1)}  + k \delta_0^{1-\alpha} +K  \right).
\end{equation}
Suppose that there is a choice of constants such that $(\ref{Acond})$ holds, then the condition  on $\omega_\psi$ in $(\ref{cEpak})$
%\begin{equation}
%\label{Abound}|\psi(\eta)-\psi(\eta-\delta)| \le  A \eta e^{-a \eta} \delta^\alpha
%\end{equation}
for  $\eta > 1$ and $\delta>0$, is renormalization invariant. This implies that for  $\eta > 1-\delta$ and $\delta>0$,
 $$|\psi(\eta)-\psi(\eta+\delta)| \le  A (\eta+\delta)  e^{-a (\eta+\delta)} \delta^\alpha <  A (1+\delta) e^{-a \delta} \eta e^{-a \eta} \delta^\alpha < A  \eta e^{-a \eta} \delta^\alpha,$$
if $a>1$ and $\delta_0=\delta_0(a)$ is sufficiently small. Using the fact that $\psi$ is an even function, it follows that 
$$|\psi(\eta)-\psi(\eta-\delta)| \le  A \eta e^{-a |\eta|} \delta^\alpha$$
for $\eta <-(1-\delta)$ and $\delta>0$, and
$$|\psi(\eta)-\psi(\eta+\delta)| \le  A \eta e^{-a |\eta|} \delta^\alpha$$
for $\eta <-1$ and $\delta>0$. We conclude that the condition on $\omega_\psi$ in the definition of the set $(\ref{cEpak})$ is invariant for all $0<\delta < \delta_0$ and $|\eta|>1$. A similar argument gives the invariance  condition for all $|\delta|<\delta_0$.

\noindent \textit{\textbf{4)  Bound on the constant $K$.}} We have computed  bounds on the norms of the  terms $L$, $J_1$, $J_2$ and $J_4$. The only remaining bound is that on $J_3$.

First, we remark, that one gets a  bound identical to  $(\ref{J3terms2})$ in the case of $\eta<-1$. Then, for $|\eta|>1$,  we can bound the factor $1-e^{-g(\eta)}$ in $(\ref{J3terms2})$ very crudely by $C (1-\beta) |\eta|$,  and  the factor $1-e^{-f(\eta)}$  as   $C(1-\beta) |\eta|^2$.
   \begin{align}
     \nonumber \|J_3\|_{p} & \le  C (1 \hspace{-0.5mm} \shortminus \hspace{-0.5mm} \beta) \delta^\alpha k \hspace{-0.8mm} \left( \hspace{-0.8mm} K^p \hspace{-3.2mm} \int \displaylimits_{|\eta| \le 1} \hspace{-2.0mm} |\eta|^p     \dif \eta \hspace{-0.5mm} + \hspace{-0.5mm} K^p \hspace{-3.2mm} \int\displaylimits_{|\eta| >1} \hspace{-2.5mm}  |\eta|^p e^{ \shortminus a p |\eta|}     \dif \eta \hspace{-0.5mm}  + \hspace{-0.5mm}  A^p \hspace{-3.2mm}\int \displaylimits_{|\eta| >1}  \hspace{-2.0mm}  |\eta|^{3 p} e^{ \shortminus a p |\eta|}     \dif \eta    \hspace{-0.8mm}  \right)^{ \hspace{-1.0mm} {1 \over p}} \\
     \label{J3} &\le   C k (1 - \beta) \  \delta^\alpha \ (K+A).
 \end{align}
We can now collect $(\ref{Lestimate2})$, $(\ref{J1})$, $(\ref{J2})$, $(\ref{J3})$ and $(\ref{J4})$ to get  that the condition $\| \omega_\psi\|_p$ is renormalization invariant if 
\begin{equation}
  \nonumber K  \ge  C k (1-\beta) \delta_0^{1-\alpha}+\beta^{{1\over p} - \alpha} K + C k^2 (1-\beta) \delta_0^{1-\alpha}+C k (1-\beta)\left( K +A \right).
\end{equation}
As $\beta \rightarrow 1$, this condition is tangent to
\begin{equation}
  \label{Kcond} K \left(\left({1 \over p} -\alpha  \right) -C k  \right)  \ge  C k  \delta_0^{1-\alpha} + C k^2 \delta_0^{1-\alpha} + C A k.
\end{equation}
For large $K$ and $A$, conditions $(\ref{Acond})$ and $(\ref{Kcond})$ are asymptotic to 
\begin{equation}
\nonumber    A \left( 1- C k -{1 +\alpha \over 2 +a} \right) \ge C k K \quad {\rm and} \quad   K \left(  \left({1 \over p} -\alpha  \right) -C k  \right)  \ge  C k A.  
\end{equation}
These two conditions have a solution if
\begin{equation}
\label{compcond} { {1 \over p} -\alpha -C k   \over C k} \ge {C k \over  1- C k -{1 +\alpha \over 2 +a} },
\end{equation}
where $C$ is a general stand in for a constant independent of $A$, $K$ or $k$, not necessarily the same in all instances. Condition $(\ref{compcond})$ can be satisfied by a choice of $k$.

\noindent \textit{\textbf{5)  Non-linear terms, case $\bm{\delta_0 > |\delta|>  \beta \delta_0}$.}} In this case
\begin{equation}
\nonumber \|\psi - \tau_\delta \psi   \|_p \le  \|\psi    \|_p +  \| \tau_\delta \psi   \|_p  \le  \left( \int_\RR e^{-a p |\eta|} \ d \eta  \right)^{1 \over p} { 2 k  \over \beta_0 \delta_0  } |\delta|, 
\end{equation}
and $K$ can be chosen to be the maximum of $K$ from Part $4)$ and  
$$ \left( \int_\RR e^{-a p |\eta|} \ d \eta  \right)^{1 \over p} { 2 k  \over \beta_0 \delta_0  }.$$
\end{proof}
We will now demonstrate that the intersection $\cN_p=\cE_{a,k,K,A,\alpha, \delta_0}^p \cap \cM_{\mu,\sigma}$ is non-empty. We have already shown in Lemma $\ref{nonempty}$  that the function  $\phi(\eta)=k \min\{1,|\eta|^\nu\} e^{- a |\eta|}$ is in $\cE_{a,k} \cap \cM_{\mu,\sigma}$ for $\nu > -1-\sigma$. The following Lemma shows that the same function belongs to $\cN_p$.
\begin{lemma}\label{nonempty2}
For any $0<k<1$, any $-1>\sigma>-2$ and any $1>\nu>-1-\sigma$, there exist $a>0$, $A>0$, $K>0$, $\delta_0>0$ and $\mu>0$,  such that the function  $\phi(\eta)=k \min\{1,|\eta|^\nu\} e^{- a |\eta|}$ is in $\cN_p=\cE_{a,k,K,A,\alpha, \delta_0}^p \cap \cM_{\mu,\sigma}$. 
\end{lemma}
\begin{proof}
  Let $\delta_0<1$ and consider first $\eta>1$. We can combine the following inequalities
  \begin{eqnarray}
    \nonumber k \left|e^{-a \eta} -e^{-a (\eta-\delta)}   \right| &\le& C k \delta e^{-a \eta}, \\
    \nonumber k \left|\eta^\nu e^{-a \eta} - \eta^\nu e^{-a (\eta-\delta)}   \right| & \le & C k \delta \eta^\nu e^{-a \eta}, \\
    \nonumber k \left|\eta^\nu e^{-a \eta} - (\eta-\delta)^\nu e^{-a \eta}   \right| & \le & C k  \delta^\nu e^{-a \eta},
  \end{eqnarray}
  to get that there exists some $A>0$, such that
  $$|\phi(\eta)-\phi(\eta-\delta)| \le A |\eta|^\nu |\delta|^\nu e^{-a \eta}.$$
  A similar bound holds for $\delta<0$ and $\eta < -1$. It is also straightforward that there is $K>0$ such that for all $|\delta| < \delta_0$
  $$\|\phi - \tau_\delta \phi  \|_p \le K \delta^\nu.$$
\end{proof}
%\medskip 
We have demonstrated the following result.
\begin{proposition}\label{exfp}
  For any $p \ge 1$ there exist $A>0$, $K>0$,  $\alpha$, $a$, $k$, $\mu$, $\beta_0$, $\sigma$  and $\delta_0$  such that for every $\beta \in (\beta_0,1)$ the operator $\cR_\beta$ has a fixed point $\psi_{\beta,p}$ in the $L^p(\RR)$-closure $\overline{\cN_p}$ of $\cN_p=\cE_{a,k,K,A,\alpha,\delta_0}^p \cap \cM_{\mu,\sigma}$.

  Additionally,  every $\overline{\cN_p}$, together with  $\psi_{\beta,p}$, is in $\cap_{n \ge 1} L^n(\RR)$.
\end{proposition}
\begin{proof} Choose $A$ and $K$ to be the maximum of those provided by Proposition $\ref{apriori2}$ and Lemma $\ref{nonempty2}$, and take the rest of the constants as those specified in Proposition $\ref{apriori2}$ and Lemma $\ref{nonempty2}$. Then $\overline \cN_p$ is non-empty and renormalization-invariant.  Minkowski inequality for the $L^p(\RR)$-norm together with the convexity of the set $\cE_{a,k}$ and $\cM_{\mu,\sigma}$  also implies that $\overline \cN_p$ is convex. Since all functions in $\cN_p$ are exponentially bounded, $\overline \cN_p \subset \cap_{n \ge 1} L^n(\RR)$.

The claim follows by Tikhonov fixed point theorem for the continous operator $\cR_\beta$ on a convex compact $\cR_\beta$-invariant set $\overline \cN_p$.
\end{proof}

\section{Limits of renormalization fixed points}\label{seclimits}
We emphasize, that existence of fixed points for the family $\cR_\beta$ does not yet imply existence of a solution to the equation $(\ref{problem1})$. Rather, one needs to show, that there is {\it one and the same} fixed point $\psi_p$ for {\it all}  $\beta \in (0,1)$. The next Lemma provides a step in this direction.
\begin{lemma}\label{powerbeta}
  Any fixed point $\phi_{\beta}$ of the operator $\cR_\beta$ is also fixed by  $\cR_{\beta^n}$ for all $n \in \NN$.
\end{lemma}
\begin{proof}
  We have at the base of induction.
  \begin{align}
    \nonumber     e^{-\beta^4 \eta^2} \phi_{\beta}(\beta^2 \eta) & = e^{-\beta^2 \eta^2} \phi_\beta(\beta \eta) + 2 \int_{\beta^2 \eta}^{\beta \eta} e^{-\zeta^2} (\phi_\beta \bullet \phi_\beta)(\zeta) \ \dif \zeta \\
    \nonumber      &  =   \hspace{-0.5mm} e^{\shortminus \eta^2}  \hspace{-0.8mm} \phi_\beta(\eta)  \hspace{-0.5mm} +  \hspace{-0.7mm} 2  \hspace{-1.0mm} \int_{\beta \eta}^{\eta}  \hspace{-2.2mm}e^{\shortminus \zeta^2}  \hspace{-0.5mm}(\phi_\beta \bullet \phi_\beta)(\zeta)  \hspace{0.2mm} \dif \zeta  \hspace{-0.5mm} +  \hspace{-0.7mm} 2  \hspace{-0.8mm} \int_{\beta^2 \eta}^{\beta \eta} \hspace{-1.9mm}  e^{\shortminus \zeta^2}  \hspace{-0.5mm}(\phi_\beta \bullet \phi_\beta)(\zeta)  \hspace{0.2mm} \dif \zeta \\
        \nonumber   &  =   e^{-\eta^2} \phi_\beta(\eta) + 2 \int_{\beta^2 \eta}^{\eta} e^{-\zeta^2} (\phi_\beta \bullet \phi_\beta)(\zeta) \ \dif \zeta.
  \end{align}
  Assume the result for $n=k$. Then
  \begin{align}
    \nonumber     e^{-\beta^{2(k+1)} \eta^2} &\phi_\beta(\beta^{k+1} \eta)   = e^{-\beta^{2k} \eta^2} \phi_\beta(\beta^{k} \eta)   + 2 \int_{\beta^{k+1} \eta}^{\beta^{k} \eta} e^{-\zeta^2} (\phi_\beta \bullet \phi_\beta)(\zeta)  \ \dif \zeta \\
    \nonumber  &  =    e^{\shortminus \eta^2} \phi_\beta( \eta) \hspace{-0.3mm}+ \hspace{-0.3mm} 2 \hspace{-0.5mm} \int_{\beta^{k} \eta}^{\eta} \hspace{-3mm}e^{\shortminus \zeta^2} \hspace{-0.5mm} (\phi_\beta \bullet \phi_\beta)(\zeta)  \hspace{0.5mm} \dif \zeta  \hspace{-0.3mm}  + \hspace{-0.3mm} 2 \hspace{-0.5mm}  \int_{\beta^{k+1} \eta}^{\beta^{k} \eta}\hspace{-3.0mm} e^{\shortminus \zeta^2} \hspace{-0.5mm} (\phi_\beta \bullet \phi_\beta)(\zeta) \hspace{0.5mm}  \dif \zeta \\
    \nonumber   & = e^{-\eta^2} \phi_\beta(\eta) + 2 \int_{\beta^{2(k+1)} \eta}^{\eta} e^{-\zeta^2} (\phi_\beta \bullet \phi_\beta)(\zeta) \ \dif \zeta.
  \end{align}
\end{proof}
We would like now to address the question of what happens to the functions $\psi_{\beta,p}$, described in Proposition $\ref{exfp}$,  as $\beta$ approaches $1$.

Consider a sequence of  $\{\psi_{\beta_i,p}  \}_{i=0}^\infty$, $\beta_i \in (\beta_0,1)$ and $\beta_i \rightarrow 1$ with $\psi_{\beta_i,p} \in \overline{\cN_p}$ where $\overline{\cN_p}$ is as in Proposition $\ref{exfp}$.  By compactness of the set $\overline{\cN_p}$, we can pass to a converging subsequence, which we will also denote   $\{\psi_{\beta_i,p}  \}_{i=0}^\infty$. Its limit $\psi_p$ is non-trivial. Since $\psi_p$ is exponentially bounded, we have that  $\psi_p \in \cap_{n \ge 1}L^n(\RR)$, and
\begin{equation}
\label{invFpsistar}\cF^{-1}[\psi_p] \in C^\infty(\RR) \cap \bigcap_{m \ge 2} L^m(\RR),
\end{equation}
(recall that for any $p \in (1,2]$, the Fourier transform is a bounded operator from $L^p$ to $L^{p'}$, where $p'$ is the H\"older conjugate of $p$).
  
Consider this  sequence $\{\psi_{\beta_i,p}\}_{i=0}^\infty$, $L^p$-convering to $\psi_p$. For any $\beta \in (0,1)$ there exists  a diverging sequence of integers $n_i$, such that $\beta_i^{n_i} \rightarrow \beta$. Indeed, for any $\epsilon>0$, a sufficient condition for $\beta_{i}^{n_i} \in (\beta-\epsilon, \beta +\epsilon)$ is
$$n_i \in \left( {\ln (\beta+\epsilon )  \over \ln \beta_{i} },   {\ln (\beta-\epsilon )  \over \ln \beta_{i} } \right).$$
The length of this interval is of the order $-\epsilon / \beta \ln \beta_{i}$. Therefore, for any $\epsilon$ there exists $I \in \NN$, such that this length is larger than $1$ for all $i \ge I$, and the interval contains some positive integer $n_i$. Therefore,
\begin{align}
  \nonumber \|\cR_\beta[\psi_{\beta_i,p}] \hspace{-0.5mm} \shortminus \hspace{-0.5mm} \psi_{\beta_i,p}\|_p & \le \left\| \cR_{\beta_i^{n_i}}[\psi_{\beta_i,p} ]  \hspace{-0.5mm} \shortminus \hspace{-0.5mm} \psi_{\beta_i,p} \hspace{-0.3mm} + \hspace{-0.7mm} \left( \hspace{-0.5mm}  e^{ \left(\beta^2\shortminus 1 \right) \eta^2 }  \hspace{-0.5mm} \shortminus \hspace{-0.5mm}  e^{ \left(\beta^{2 n_i}_i\shortminus 1 \right) \eta^2 }  \hspace{-0.5mm} \right) \hspace{-0.6mm} \psi_{\beta_i,p}   \right\|_p \hspace{-0.5mm} +  \\
  \nonumber & \hspace{20.0mm} \hspace{5mm} + \left\| 2 e^{\eta^2} \int_{\eta \over \beta}^{\eta \over \beta^{n_i}_i} e^{-\zeta^2} (\psi_{\beta_i,p} \bullet  \psi_{\beta_i,p})(\zeta) \ \dif \zeta   \right\|_p. 
\end{align}
According to Lemma $(\ref{powerbeta})$, we have $\cR_{\beta_i^{n_i}}[\psi_{\beta_i,p} ]=\psi_{\beta_i,p}$, therefore, 
  \begin{align}
    \nonumber \|\cR_\beta[\psi_{\beta_i,p}]  \hspace{-0.5mm}   \shortminus  \hspace{-0.5mm}  \psi_{\beta_i,p}\|_p & \le C |\beta_i^{n_i}\hspace{-0.5mm}  \shortminus \hspace{-0.3mm} \beta| \hspace{-0.5mm}  \left( \hspace{-1.5mm} \left( \int_\RR \hspace{-1.0mm}  |\eta|^{2 p} e^{\shortminus a p |\eta|  }  \ \dif \eta \right)^{\hspace{-0.8mm}{1 \over p}} \hspace{-1.7mm}  + \hspace{-0.7mm}  \left(  \int_\RR  \hspace{-1.5mm}  e^{\shortminus a p { |\eta| \over \hat \beta_0} } |\eta|^{2 p} \ \dif \eta  \right)^{\hspace{-0.8mm}{1 \over p}} \hspace{-0.4mm} \right) \\
    \nonumber &\le  C |\beta_i^{n_i}-\beta|.
\end{align}
We conclude with the following
\begin{proposition}
For every $p \ge 1$ there exists $\psi_p \in \overline \cN_p \subset \cap_{n \ge 1} L^n(\RR)$,  which is a fixed point of $\cR_\beta$ for all $\beta \in (0,1)$.  
\end{proposition}

\section{More on the common fixed point}\label{secfp} We can, however, go one step further and derive an exact form of the limit function $\psi_p$. First,    
  \begin{align}
\nonumber  e^{-\beta^{2}_i \eta^2} \hspace{-0.8mm}\psi_{\beta_i,p}(\beta_i \eta) \hspace{-0.4mm} - \hspace{-0.4mm} e^{-\eta^2} \hspace{-0.8mm}\psi_{\beta_i,p}(\eta)&= 2 \int_{\beta_i \eta}^{\eta} e^{-\zeta^2} (\psi_{\beta_i,p} \bullet \psi_{\beta_i,p})(\zeta) \ \dif \zeta \\
  \label{eq_diff1}   &=  (1\hspace{-0.4mm}-\hspace{-0.4mm}\beta_i) \hspace{-1mm} \left( \hspace{-0.2mm}  2 e^{-\eta^2} \hspace{-1.1mm} \eta (\psi_{\beta_i,p} \bullet \psi_{\beta_i,p})(\eta)\hspace{-0.4mm}+ \hspace{-0.4mm}  \eta I(\eta,\beta_i) \hspace{-0.7mm} \right)\hspace{-0.6mm},
  \end{align}
where, according to Lebesgue differentiation theorem, the remainder $I$
\begin{equation}
  \nonumber  \left|I(\eta,\beta_i)\right|= 2{1 \over (1-\beta_i) |\eta|}\left|\int_{\beta_i \eta}^{\eta} e^{-\zeta^2} \psi_{\beta_i,p} \bullet \psi_{\beta_i,p} (\zeta) -  e^{-\eta^2}  \psi_{\beta_i,p} \bullet \psi_{\beta_i,p}(\eta) \ \dif \zeta \right|
\end{equation}
converges to zero a.e. as $\beta_i\rightarrow 1$. In fact, by the Lebesgue dominated convergence theorem, $I(\eta,\beta_i)$ converges to zero in $L^p$. Indeed, $|I(\eta,\beta_i)|$ can be bounded by the function $2 e^{-\beta_0^2 \eta^2 - a \beta_0 |\eta|} |\eta| \in L^p$, with $a$ as in Proposition $\ref{bounds}$. % while  $\left|I(\eta,\beta_i)\right| \stackMath\brlap[1.4ex]{\scriptstyle \beta_i \rightarrow  1} \longrightarrow  0$.
%\begin{equation*}
% \left|I(\eta,\beta_i)\right|= 2\left|e^{-\beta_* \eta^2} \psi_{\beta_i,p} \bullet \psi_{\beta_i,p} (\beta_* \zeta) -  e^{-\eta^2}  \psi_{\beta_i,p} \bullet \psi_{\beta_i,p}(\eta)  \right| \stackMath\brlap[1.4ex]{\scriptstyle \beta_i \rightarrow  1} \longrightarrow  0.
%\end{equation*}
We, therefore, have that the following functions are in $L^p(\RR)$:
\begin{equation*}
F_p(\eta,\beta_i)= {e^{-\beta^{2}_i \eta^2} \psi_{\beta_i,p}(\beta_i \eta) - e^{-\eta^2} \psi_{\beta_i,p}(\eta) \over (1-\beta_i) \eta},
\end{equation*}
and $F_p(\eta,\beta_i) - 2 e^{-\eta^2} (\psi_{\beta_i,p} \bullet \psi_{\beta_i,p})(\eta)$ converges to $0$ in  $L^p(\RR)$ as $\beta_i \rightarrow 1$.

We claim that the sequence  $F_p(\eta,\beta_i)$ converges to the weak derivative of the sequence $\phi_{\beta_i,p}(\eta)= -e^{- \eta^2} \psi_{\beta_i,p}(\eta)$. Indeed, given $v \in C_0^\infty(\RR)$
%\begin{equation}
%  \nonumber \int \displaylimits_\RR \hspace{-1.8mm} F_p \hspace{-0.3mm}(\eta,  \hspace{-0.1mm}\beta) v(\eta)  \dif \eta=\hspace{-2mm} \int \displaylimits_\RR  \hspace{-1.5mm} \phi_{\beta}  \hspace{-0.3mm}(\eta){  \hspace{-0.8mm}v(\eta) \hspace{-0.5mm} \shortminus \hspace{-0.5mm}   v \hspace{-0.6mm}\left( \hspace{-0.3mm} {\eta \over \beta}  \hspace{-0.3mm} \right) \hspace{-1.8mm}  \over (1  \hspace{-0.5mm} \shortminus \hspace{-0.5mm}   \beta) \eta }  \hspace{0.2mm}  \dif  \eta   \hspace{0.2mm} \shortminus \hspace{-0.9mm}   \int \displaylimits_\RR  \hspace{-1.5mm} \phi_{\beta}  \hspace{-0.3mm}(\beta \eta){v(\eta)   \hspace{-1.6mm} \over (1   \hspace{-0.5mm} \shortminus \hspace{-0.5mm}  \beta) \eta }   \dif \eta    + \hspace{-0.7mm}  \int \displaylimits_\RR  \hspace{-1.5mm} \phi_{\beta} \hspace{-0.3mm}(\eta){v \hspace{-0.6mm}\left(  \hspace{-0.3mm} {\eta \over \beta}  \hspace{-0.3mm} \right)  \hspace{-1.6mm} \over (1  \hspace{-0.5mm} \shortminus \hspace{-0.5mm}   \beta) \eta}    \dif \eta.
%\end{equation}
\begin{align}
  \nonumber \int \displaylimits_\RR  F_p(\eta,\beta) v(\eta)  \dif \eta & =  \int_\RR \phi_{\beta,p}(\eta){v(\eta) -  v\left({\eta \over \beta} \right)  \over (1  -   \beta) \eta } \ \dif  \eta   -  \int_\RR   \phi_{\beta,p}(\beta \eta){v(\eta)   \over (1  - \beta) \eta }  \  \dif \eta  +\\
  \nonumber & \hspace{49.6mm} +  \int_\RR \phi_{\beta,p}(\eta){v\left({\eta \over \beta} \right) \over (1  -  \beta) \eta}  \  \dif \eta.
\end{align}
The last two integrals cancel after a change of variables. Therefore,
\begin{eqnarray}
  \nonumber \int_\RR F_p(\eta,\beta) v(\eta) \ \dif \eta =-\int_\RR \phi_{\beta,p}(\eta)v'(\eta) \ \dif \eta + \int_\RR \phi_{\beta,p}(\eta)O((1-\beta) \eta) \ \dif \eta.
\end{eqnarray}
Since $\phi_\beta,p$ is exponentially bounded, the last integral is $O(1-\beta)$. Therefore, for any $p \ge 1$ the points $\psi_{\beta_i,p}$ $L^p$-converge to a weak $L^{p}$-solution of 
\begin{equation}
\label{limiteq}   \vartheta'(\eta)=2  \eta \vartheta(\eta)- 2 \vartheta \bullet \vartheta(\eta).
\end{equation}
Since any weak $L^p$-solution of equation $\ref{limiteq})$ is necessarily differentiable, it is enough to consider only classical solutions of   $\ref{limiteq})$.  This is a Riccati equation after the Laplace transform $\hat \vartheta(s)=\cL[\vartheta](s)$:% (whose region of convergence is $\Re (s)>-a$ with $a$ as in Proposition $\ref{bounds}$).
\begin{equation}
\label{Riccati} \hat \vartheta'(s)=-\hat \vartheta(s)^2-{s \over 2} \hat \vartheta(s)+{\nu \over 2},
\end{equation}
where $\nu:=\vartheta(0)=\psi_p(0)$. Equation $(\ref{Riccati})$ has a trivial solution when $\nu=0$. When $\nu=1$, its has a solution $\hat \vartheta(s)=1/s$, corresponding to the solution $\vartheta(\eta)=1$ of $(\ref{limiteq})$. The substitution $\hat \vartheta(s)= v'(s)/v(s)$ reduces this to a second order linear differential equation, and a further change of variables $v(s)=e^{-{s^2 \over 4}}g(s)$, reduces it to a Kummer's equation for the function $u(s^2/4)=g(s)$:
$$z u''(z)+\left({1 \over 2} -z  \right) u'(z) -{1+\nu \over 2} u(z)=0,$$
with the general solution
$$u(z)=c_1 M\left(s_1, {1 \over 2},z \right)+c_2 U\left(s_1, {1 \over 2},z \right),$$
where $s_1=(1+\nu)/2$,  $M$ is the confluent hypergeometric function, and $U$ is Tricomi's confluent hypergeometric function:
\begin{equation*}
 \nonumber  U\left(s_1, {1 \over 2},z \right)={\sqrt{\pi} \over \Gamma(s_2)}  M\left(s_1, {1 \over 2},z \right)-{2 \sqrt{\pi} \over  \Gamma (s_1)} \sqrt{z} M\left(s_2, {3 \over 2}, z\right),
  \end{equation*}
where  $s_2=(2+\nu)/2$. In particular,
\begin{equation}
  \label{hatvartheta} \hat \vartheta(s)={d \over d s} \ln \left( c_1  e^{-{s^2 \over 4}} M\left(s_1, {1 \over 2},{s^2 \over 4} \right)+c_2 e^{-{s^2 \over 4}} U\left(s_1, {1 \over 2},{s^2 \over 4} \right)    \right)
\end{equation}
is a  one-parameter family of solutions of   $(\ref{Riccati})$, the parameter being $(c_1,c_2) \in \RRPP^1$. Taking the inverse Laplace transform of $(\ref{hatvartheta})$ and equating its value at $0$ with $\vartheta(0)$, one obtains the the parameter $(c_1,c_2)$.

We consider solutions for $\nu <1$.
\begin{equation}
\label{hatvartheta2}  \hat \vartheta(s)(s)={d \over d s} \ln  \hspace{-0.5mm} \left(  \hspace{-0.5mm}  c_1 e^{-{s^2 \over 4}} M  \hspace{-0.8mm} \left( \hspace{-0.5mm} {1+\nu \over 2}, {1 \over 2},{s^2 \over 2}  \hspace{-0.5mm} \right)+c_2  e^{-{s^2 \over 4}} U  \hspace{-0.8mm} \left(  \hspace{-0.5mm} {1+\nu \over 2}, {1 \over 2},{s^2 \over 2}  \hspace{-0.5mm} \right)  \hspace{-0.5mm} \right) \hspace{-0.5mm}.
  \end{equation}
Since the asymptotics of function $M$ is
\begin{equation}
\label{asympt} M(a,b,z) \sim \Gamma(b) \left( { e^z z^{a-b} \over \Gamma(a)} + {(-z)^{-a} \over \Gamma(b-a)} \right),
\end{equation}
we get that for large positive $s$,
$$ e^{-{s^2 \over 4}} M\left({1+\nu \over 2}, {1 \over 2},{s^2 \over 2} \right) \sim {\rm const} \ s^{\nu},$$
and similarly for $U$. Therefore, the asymptotics of $\hat \vartheta(s)$ is $\nu/s$, which is bounded for large $s$ from above by $1/(s+a)$ for some positive $a$, and $\hat \vartheta$ is indeed a Laplace transform of an exponentially bounded function.

We conclude that
\begin{equation}
  \label{varthetastar}  \vartheta_{\nu}(\eta)= \cL^{-1}[\hat \vartheta](\eta),
\end{equation}
where $\cL^{-1}$ is the inverse Laplace transform, is a family of solutions of  equation $(\ref{limiteq})$ such that $\vartheta_{\nu} \in \overline{\cN_p}$ for some $\nu \in (0,1)$.% depending on the parameter $\nu$ (alternatively, $(c_1,c_2) \in \RR P^1$).

This concludes the prof of the Main Theorem.

We emphasize, however, that the computation of the inverse Laplace transform for the function $(\ref{hatvartheta2})$ is a non-trivial undertaking which can not be performed in quadratures.

\bibliographystyle{plain}
\bibliography{biblio}

\begin{thebibliography}{10}

\bibitem{BLM}
G.~R. Baker, X.~Li, and A.~C. Morlet.
\newblock Analytic structure of two $1d$-transport equations with nonlocal
  fluxes.
\newblock {\em Physica D}, 91:349--375, 1996.

\bibitem{BDS}
Yu. Bakhtin, E.~I. Dinaburg, and Ya.~G. Sinai.
\newblock On solutions with infinite energy and enstrophy of the navier-stokes
  system.
\newblock {\em Russian Math. Surveys}, 59(6):55--72, 2004.

\bibitem{CCCF}
D.~Chae, C\'ordoba. A., D.~C\'ordoba, and M.~A. Fontelos.
\newblock Finite time singularities in a 1d model of the quasi-geostrophic
  equation.
\newblock {\em Adv. Math.}, 194:203--223, 2005.

\bibitem{CLM}
P.~Constantin, P.~D. Lax, and A.~Majda.
\newblock A simple one-dimensional model for the three-dimensional vorticity
  equation.
\newblock {\em Commun. Pure Appl. Math.}, 38:1251--1263, 1985.

\bibitem{DG}
S.~De~Grigorio.
\newblock On a one-dimensional model for the three-dimensional vorticity
  equation.
\newblock {\em J. Stat. Phys.}, 59:1251--1263, 1990.

\bibitem{HHM}
H.~Hanche-Olsen, H.~Holden, and E.~Malinnikova.
\newblock An improvement of the {K}olmogorov-{R}iesz compactness theorem.
\newblock {\em Expo. Math}, 37(1):84--91, 2019.

\bibitem{Leray}
J.~Leray.
\newblock Essai sur le mouvement d'un fluide visqueux emplissant l'espace.
\newblock {\em Acta Math.}, 63:193--248, 1934.

\bibitem{LS1}
D.~Li and Ya.~G. Sinai.
\newblock Blowups of complex solutions of the $3$d-navier-stokes system and
  renormalization group method.
\newblock {\em Journal of European Mathematical Society}, 10(2):267--313, 2008.

\bibitem{LS0}
D.~Li and Ya.~G. Sinai.
\newblock Complex singularities of solutions of some 1d hydrodynamic models.
\newblock {\em Physica D}, 237:1945--1950, 2008.

\bibitem{LS3}
D.~Li and Ya.~G. Sinai.
\newblock Blowups of complex-valued solutions for some hydrodynamic models.
\newblock {\em Regul. Chaotic Dyn.}, 15(4-5), 2010.

\bibitem{LS2}
D.~Li and Ya.~G. Sinai.
\newblock Singularities of complex-valued solutions of the two-dimensional
  burgers system.
\newblock {\em J. Math. Phys.}, 51:521--531, 2010.

\bibitem{MajdaB02}
A.~J. Majda and A.~L. Bertozzi.
\newblock {\em Vorticity and Incompressible Flow}.
\newblock Cambridge University Press, 2002.

\bibitem{Mats}
Y.~Matsuno.
\newblock Linearization of novel nonlinear diffusion equations with the hilbert
  kernel and their exact solutions.
\newblock {\em J. Math. Physics}, 32:120, 1991.

\bibitem{NRS}
J~Ne\v{c}as, M.~R\r{u}\v{z}i\v{c}ka, and V~\v{S}ver\'ak.
\newblock On leray’s self-similar solutions of the navier–stokes equations.
\newblock {\em Acta Math.}, 176:283--294, 1996.

\bibitem{Schochet}
S.~Schochet.
\newblock Explicit solutions of the viscous model vorticity equation.
\newblock {\em Commun. Pure Appl. Math.}, 39:531--535, 1986.

\bibitem{Tsai}
T-P. Tsai.
\newblock On leray's self–similar solutions of the navier–stokes equations
  satisfying local energy estimates.
\newblock {\em Arch. Ration. Mech. Anal.}, 143:29--51, 1998.

\end{thebibliography}
\end{document}